\documentclass[11pt, reqno]{amsart}
 \usepackage{graphicx}               
\usepackage{amsmath}
\usepackage{a4,amsthm}
\usepackage{mathabx}
\usepackage{amssymb} 

\theoremstyle{plain}
\newtheorem{theorem}{Theorem}[section]
\newtheorem{lem}[theorem]{Lemma}

\newtheorem{prop}[theorem]{Proposition}
\newtheorem{cor}[theorem]{Corollary}

\newtheorem{remark}[theorem]{Remark}
\newtheorem{exa}[theorem]{Example}

\def\Ad{{\rm Ad}}

\def\a{{\mathfrak{a}}}

\def\s{{\mathfrak{s}}}

\def\m{{\mathfrak{m}}}

\def\p{{\mathfrak{p}}}    
\def\k{{\mathfrak{k}}}

\def\g{{\mathfrak{g}}}
\def\l{{\mathfrak{l}}}
\def\h{{\mathfrak{h}}}

\def\Z{{\mathbb{Z}}}

\def\C{{\mathbb{C}}}
\def\R{{\mathbb{R}}}

\def\nsmallskip{\smallskip\noindent}
\def\bbigskip{\bigskip\bigskip}

\def\nmedskip{\medskip\noindent}

\def\buildunder#1#2{\mathrel{\mathop{\kern0pt #2}
\limits_{#1}}}

\def\dds{\frac{d}{ds}{\big |_{s=0}}}
\def\ddt{\frac{d}{dt}{\big |_{t=0}}}

\def\pn{\par\noindent}
\def\sn{\smallskip\noindent}
\def\mn{\medskip\noindent}
\def\bn{\bigskip\noindent}
 \def\REM #1{}

\begin{document}

\title[LOGCVX vs. PSH]{Invariant plurisubharmonic
functions on  non-compact Hermitian symmetric spaces.}

\bbigskip

\author[Laura Geatti]{Laura Geatti}
\author[Andrea Iannuzzi]{Andrea Iannuzzi}

\address{Dipartimento di Matematica,
Universit\`a di Roma  ``Tor Vergata", Via della Ricerca Scientifica 1,
I-00133 Roma, Italy} 
\email{geatti@mat.uniroma2.it, iannuzzi@mat.uniroma2.it}

\thanks {\ \ {\it Mathematics Subject Classification (2010):} 32M15, 31C10, 32T05}

\thanks {\ \ {\it Key words}: Hermitian symmetric spaces, Stein domains, plurisubharmonic functions}

\thanks {\ \  The authors acknowledge the MIUR Excellence Department Project awarded to the Department of Mathematics, University of Rome Tor Vergata, CUP
E83C18000100006
}

\begin{abstract}
Let $\,G/K\,$ be an irreducible Hermitian symmetric space and
let $\,D\,$ be a $\,K$-invariant domain in $\,G/K\,$. In this paper we characterize several classes of  $\,K$-invariant plurisubharmonic functions on $\,D\,$ in terms of their restrictions to a  slice   intersecting all $\,K$-orbits. As applications  we show that $\,K$-invariant plurisubharmonic functions on $\,D\,$ are necessarily continuous and we reproduce
the classification of Stein $\,K$-invariant domains
in $\,G/K\,$ obtained by E. Bedford and J. Dadok in \cite{BeDa91}.

 \end{abstract} 
\maketitle


\section{Introduction} 
\bigskip

Let $\,G/K\,$ be an irreducible Hermitian symmetric space of rank $\,r$. 
By the polydisk theorem (cf. [Wo72],  p.280) the space $\,G/K\,$ contains a closed subspace $\,\Delta^r$, biholomorphic to an $\,r$-dimensional   
polydisk, with the property that $\,G/K=K\cdot \Delta^r$.
If $\,D\,$ is a $\,K$-invariant domain in $\,G/K$, one has $\,D=K\cdot R$, where $\,R:=D\cap \Delta^r\,$   is a Reinhardt domain in $\,\Delta^r$. The polydisk $\,\Delta^r\,$ and $\,R\,$ are invariant under the group
$\,T\ltimes {\mathcal S}_r$, generated by   rotations and   coordinate permutations.

As the Reinhardt domain $\,R\,$ intersects all $\,K$-orbits in $\,D$, it encodes all information on $\,K$-invariant objects in $\,D$. 
In this paper we focus on $\,K$-invariant plurisubharmonic functions  on    $\,D$.  When $\,D\,$ is  Stein, we obtain the following   characterization of the class
 $\,P^{{\infty},+}(D)^K\,$ of  smooth,  $\,K$-invariant, strictly plurisubharmonic functions on $\,D\,$:
 
\nmedskip
\centerline {\it $\,f\,\in\,P^{{\infty},+}(D)^K\ \ \ $ if and only if $\ \ \  f|_R\, \in\,P^{{\infty},+}(R)^{T \ltimes \mathcal S_r}$\,,}
  
  \mn
 where  $\,f|_R\,$ is the restriction of $\,f\,$ to $\,R$.
Such  result is  extended to wider classes of plurisubharmonic functions
as follows. Let $\,P^{\infty}(D)^K\,$ denote the class
of smooth, $\,K$-invariant,  plurisubharmonic  functions and 
$\,P^+(D)^K\,$  (resp. $\,P(D)^K\,$) the class of $\,K$-invariant, strictly
plurisubharmonic (resp.  plurisubharmonic) functions on~$\,D$.
One has (Thm. \ref{PLURIREIN}):

\nmedskip
{\bf Theorem.}
{\it
The restriction map $\,f\to  f|_R\,$ is a bijection between 
\begin{itemize}
\item [(i)] $\,P^{{\infty},+}(D)^K\,$ and  $\,\ \ \ P^{{\infty},+}(R)^{T \ltimes \mathcal S_r}$,
\item [(ii)] $\,P\,(D)^K\,\ \ \ \,$ and $\,\ \ \ P\,(R)^{T \ltimes \mathcal S_r}$,
\item [(iii)] $\,P^{\infty}(D)^K\,\ \ $ and $\,\ \ \ P^{\infty}(R)^{T \ltimes 
\mathcal S_r}$,
\item [(iv)] $\,P^{+}(D)^K\,\ \ \,$ and $\,\ \ \ P^{+}(R)^{T \ltimes 
\mathcal S_r}$.
\end{itemize}}

\medskip

As a by-product  (Cor. \ref{CHARACTER}) we reproduce  the classification of Stein $\,K$-invariant domains
in $\,G/K\,$  obtained  by Bedford and Dadok in \cite{BeDa91} 
  by a direct computation on some classical cases (see also \cite{FeHu93}).

\nmedskip
 {\bf Corollary}.
 {\it
Let $\,D\,$ be a $\,K$-invariant domain in  $\,G/K$.
 \begin{itemize}
\item [(i)]  If $\,G/K\,$ is of tube type, then $\,D\,$ is Stein if and only if $\,R\,$ is Stein and connected.
\item [(ii)] If $\,G/K\,$ is not of tube type, then $\,D\,$ is Stein if and only if
 $\,R\,$ is Stein and complete. In particular $\,R\,$ contains the origin and  is connected.  
  \end{itemize}
  }

Let
$\,\g= \k\oplus \p\,$ be  a Cartan decomposition  of the Lie algebra $\,\g\,$
of $\,G$, let $\,\a \,$ be a maximal abelian subspace of $\,\p$, with Weyl group $\,W$,  and let 
$\,G=K\exp \a \,K\,$ be the corresponding decomposition of $\,G$. Every $\,K$-invariant domain
 $\,D\,$ in $\,G/K\,$ is uniquely determined by a $\,W$-invariant domain $\, \mathcal D_\a\,$ in $\,\a\,$ by
$$\mathcal D_\a \to K \exp \mathcal D_\a\,K/K=D\,.$$
Moreover, to  every smooth $\,K$-invariant function $\,f\,$ on  $\,D\,$ there corresponds a unique 
 smooth $\,W$-invariant function $\,\tilde f\,$ on  $\,\mathcal D_\a$,
  defined by $$\, \tilde f(H):= f (\exp (H) K)\,,\quad H\in\ \mathcal D_\a,$$  (cf.  \cite{Fle78}, \cite{Dad82}).
  The proof of the above results is  carried out as follows. 
As a first step we explicitly compute the Levi form
 of $\,f\,$ in terms of the first and second derivatives of  $\,\tilde f$. This is achieved in
  Proposition \ref{LEVI} by means of a fine decomposition of the tangent bundle of $\,D$, induced by the restricted root decomposition of $\,\g\,$ (Rem. \ref{BASIS2}), and
 a  simple pluripotential argument which enable us to maximally exploit the involved symmetries.

The Levi form  computation leads to the key ingredient of our results. Namely, 
the following  characterization of
smooth  $\,K$-invariant strictly plurisubharmonic functions on a Stein $\,K$-invariant domain $\,D$ (Thm.\,\ref{BIJECTIVEK1}):

 \nmedskip
 \centerline
{\it  $\,f\,\in \,P^{{\infty},+}(D)^K\ \ \ $ if and only if $\ \ \  \tilde f\, \in\,LogConv^{\infty,+}(\mathcal D_\a)^{W}$,}
 
 \nmedskip
where the latter class consists of  smooth $\,W$-invariant functions on $\,\mathcal D_\a\,$ 
satisfying an appropriate differential positivity condition. We also show that 
$\,\tilde f\,$ belongs to $\,LogConv^{\infty,+}(\mathcal D_\a)^{W}\,$ if and only if the corresponding $\,T \ltimes \mathcal S_r$-invariant function on $\,R\,$ 
is strictly plurisubharmonic (Prop.\,\ref{PSHPOLIDISC}). This fact, which may be of independent interest 
in the context of Reinhardt domains, implies (i) in the above theorem.   
When  extending the above  characterization  to the non-smooth setting (Thm.\,\ref{BIJECTIVEK}),  it turns out that  $\,K$-invariant plurisubharmonic functions on $\,D\,$ are necessarily continuous.

Finally, in the appendix we explicitly determine a $\,K$-invariant potential
of the Killing form on $\,G/K\,$ (Prop. \ref{POTENTIALK}) and we
observe  that, up to an additive constant, it coincides
with the logarithm of the Bergman kernel function. 

We remark that  our methods require no classification results, nor any distinction between classical and exceptional cases. 

We wish to thank our colleague Stefano Trapani for several useful
discussions and 
suggestions.


\section{Preliminaries}
\label{PRELIMINARIES}

\bigskip

Let $\,\g\,$ be a non-compact semisimple Lie algebra and let $\,\k\,$ be a maximal compact subalgebra of $\,\g\,$.
 Let $\,\g=\k\oplus\p$ be the Cartan decomposition of $\g$ with respect to $\,\k$, with  Cartan involution $\theta$.
Let  $\,\a\,$ be a maximal abelian subspace in $\,\p$. The dimension $r$ of $\,\a\,$ is by definition the {\it rank} of
$\,G/K$.
Let 
 $\g= \m\oplus \a\oplus\bigoplus_{\alpha\in\Sigma}\g^\alpha $ be the restricted root decomposition of $\g$, 
 where $\,\m\,$ is the centralizer of $\,\a\,$ in $\,\k$,
 the joint eigenspace 
 $\,\g^\alpha=\{X\in\g~|~ [H,X]=\alpha(H)X, {\rm \ for\  all \ } H\in\a\}\,$ is the $\,\alpha$-restricted root space and 
the restricted root system $\,\Sigma\,$ consists of those  $\,\alpha\in\a^*\,$ for which
 $\,\g^\alpha\not=\{0\} $. Denote by $B(\,\cdot\,,\,\cdot\,)$ the Killing form of $\g$, as well as its holomorphic extension to $\g^\C$ (which coincides with the Killing form of $\g^\C$). Denote by $W$ the Weyl group of $\a$, i.e. the quotient of the normalizer over the centralizer of $\a$ in $K$. In the Hermitian case $W$ acts on $\a$ by signed permutations.

For $\,\alpha\in\Sigma$, consider the $\,\theta$-stable  space $\,\g[\alpha]:=\g^\alpha\oplus  \g^{-\alpha}$, and denote by
$\,\k[\alpha]\,$ and
$\,\p[\alpha]\,$ the projections of $\,\g[\alpha]\,$ along $\,\p\,$ and $\,\k$, respectively. Let $\Sigma^+$ be a choice of positive roots in $\Sigma$. 
 Then  
\begin{equation}\label{DECO}\k=\m \oplus  \bigoplus_{\alpha \in \Sigma^+} \k[\alpha]\, \qquad 
{\rm and} \qquad \p= \,\a  \oplus \bigoplus_{\alpha \in \Sigma^+} \p[\alpha] \, \end{equation}
are $\,B$-orthogonal decompositions of $\,\k\,$ and $\,\p$, respectively. The next lemma is an easy exercise.
 
\begin{lem}
\label{BASIS}
Every element $\,X\,$ in $\,\p\,$ decomposes in a unique way  as 
 $$X_\a + \textstyle \sum_{\alpha \in \Sigma^+} P^\alpha,$$
where $\,X_\a\in \a\,$ and    $\,P^\alpha\in \p[\alpha]$. 
The vector $\,P^\alpha\,$ can be written uniquely as  $\,P^\alpha=X^\alpha-\theta X^\alpha$, where
$\,X^\alpha\,$ is the component of $\,X\,$ in the root space $\,\g^\alpha$.
Moreover,  $[H,P^\alpha]=\alpha(H)K^\alpha$, where $\,K^\alpha\,$ is the element in $\,\k[\alpha]\,$ uniquely defined by
$\,K^\alpha=X^\alpha+\theta X^\alpha$.
\end{lem}

The restricted root system of a Lie algebra $\,\g\,$  of Hermitian type is either of type 
$\,C_r\,$ (if $\,G/K\,$ is of tube type) or of type $\,BC_r\,$ (if $\,G/K\,$ is not of tube type), i.e.  there exists a basis
$\,\{e_1,\ldots,e_r\}\,$ of $\,\a^*\,$ for which   
$$\Sigma^+=\{2e_j, ~1\le j\le r,~~e_k\pm e_l,~ 1\le k< l\le r\},\quad \hbox{ for type
$\,C_r$}, $$
$$\Sigma^+=\{e_j,~2e_j,~1\le j\le r,~~e_k\pm e_l,~~1\le k<l\le r\},\quad \hbox{ for type
$\,BC_r$}\,.$$
With the above choice of a positive system $\,\Sigma^+$, the roots
$$ 2e_1,\,\dots\,,\,\ 2e_r \,$$
form a maximal set of  long strongly orthogonal positive restricted  roots (i.e. such that  $\,2e_k\pm 2e_l\not\in
\Sigma $, for $\,k\not=l$).

Denote by $\,I_0\,$ the $\,G$-invariant  complex structure  of $\,G/K$.
For every $\,j=1,\ldots,r$, the root space   $\, \g^{2e_j}\,$ is one-dimensional. Choose
generators $\,\ E^j \in \g^{2e_j}\,$ such  that 
 the $\,\s \l (2)$-triples
$\,\{E^j,~\theta  E^j,~ A_j:=[\theta  E^j,\, E^j]\}\,$
are normalized as follows  
\begin{equation}\label{NORMALIZ1}
[ A_j,\, E^j]=2  E^j, \quad \hbox{for}\quad j=1,\ldots,r. \end{equation}
We also assume that $I_0(E^j-\theta E^j) = A_j$
(see \cite{GeIa13}, Def. 2.1). 
By the strong orthogonality of $ 2e_1,\ldots, 2e_r$,  the  vectors $\, A_1,\ldots, A_r \,$ form  a $\,B$-orthogonal basis of $\,\a\,$, dual to the basis $e_1,\ldots,e_r$ of $\a^*$,  and 
the  associated $\,\s \l (2)$-triples pairwise commute.  
For $j=1,\ldots,r$, define 
\begin{equation} \label{KJPJ} K^j:=E^j+\theta E^j\quad\hbox{ and}\quad  P^j:=E^j-\theta E^j.\end{equation}

 On $\,\p \cong T_{eK} G/K$ the complex structure $I_0$ coincides 
with the adjoint action of the element $\,Z_0 \in Z(\k)\,$ given by 
\begin{equation}
\label{CENTER}
 \textstyle Z_0=S_0+\frac{1}{2} \sum_{j=1}^r K^j \,,
 \end{equation}
for some element   $\,S_0\,$ in  a Cartan subalgebra $\,\s\,$  of $\m$. One has $\,S_0=0\,$ in the tube case (see \cite{GeIa13}, Lem. 2.2). 
The complex structure $\,I_0\,$ permutes the blocks of the decomposition (\ref{DECO}) of~$\,\p$ (cf. \cite{RoVe73}), namely 
\begin{equation}
\label{CPLXBIS}
\,I_0\a=\bigoplus_{j=1}^r \p[2e_j], \quad I_0\p[e_j+e_l]=\p[e_j-e_l],  \quad I_0 \p[e_j]=\p[e_j]\,.
 \end{equation}

In order to state the next result, we need to recall a few more  facts.  Let $ \g^\C=\h^\C\oplus \bigoplus_{\mu\in\Delta} \g^\mu$ be the root decomposition of $\g^\C$ with respect
to  the maximally split Cartan subalgebra $\h=\s\oplus \a$ of $\g$. Let $\sigma$ be the conjugation
of $\g^\C$ with respect to $\g$. Let $\theta$ denote also the $\C$-linear extension of 
$\theta$ to $\g^\C$. One has $\theta\sigma=\sigma\theta$. Write $\overline Z:=\sigma Z$, for $Z\in\g^\C$. As  $\sigma$ and $\theta$ stabilize  $\h$, they induce   actions on $\Delta$, defined by
$\bar\lambda (H):=\overline{\lambda(H)}$ and $\theta\lambda(H):=\lambda(\theta(H))$, for $H\in\h$, respectively.  Fix a positive root system  $\Delta^+$ compatible with 
$\Sigma^+$, meaning that $\lambda|_\a =Re(\lambda) \in \Sigma^+$ implies $\lambda\in\Delta^+$. Then 
$\sigma\Delta^+= \Delta^+$.

\bn

The next  lemma gives a more detailed description of the complex structure $I_0$ on~$\p$.

\begin{lem} \label{COMPLEXSTRUCTURE}
\item{$(a)$}  For $j=1,\ldots,r$, let $A_j$ and $P^j$ be as in $($\ref{NORMALIZ1}$)$ and $($\ref{KJPJ}$)$.   One has $I_0P^j =A_j$ and $I_0A_j=-P^j$.  \\

\item{$(b)$}  Let $P=X-\theta X\in\p[e_j+e_l]$, where $X=Z^\mu+\overline Z^\mu$, with $Z^\mu\in\g^\mu$ and $\mu\in\Delta^+$ is a root satisfying 
  $Re(\mu)=e_j+e_l$. If $\bar\mu=\mu$, we may assume $Z^\mu=\overline Z^\mu$
  and set $X=Z^\mu$.
Then $I_0P=Y-\theta Y$, where   $Y=[Z_0,X]={1\over 2} [K^j+K^l,X]\in \g^{e_j-e_l}
\oplus  \g^{\theta(e_j-e_l)}$. \\

\item{$(c)$} Let $P=X-\theta X\in\p[e_j]$, where $X=Z^\mu+\overline Z^\mu$, $Z^\mu\in\g^\mu$ and $\mu$ is a root in
$\Delta^+$ satisfying  $Re(\mu)=e_j$ $($as $\dim \p[e_j]$ is even, necessarily
$\bar\mu\not= \mu)$. Then $I_0P=Y-\theta Y$, where  $Y=i(Z^\mu- \overline{Z^\mu})\in\g^{e_j}$.
\end{lem} 

\begin{proof} 
(a)  follows by definition from (\ref{NORMALIZ1}) and  (\ref{KJPJ}). 

\noindent
Observe that  $Z_0 \in Z(\k)$ implies   
\begin{equation} \label{AA}[Z_0,X]=-[Z_0, \theta X], \quad \hbox{  for every $X\in \g$}.
\end{equation} 
\sn
(b) By (\ref{CENTER}), (\ref{CPLXBIS})  and the fact that $[S_0,\g^\mu]\subset \g^\mu$,  for every  $\mu\in \Delta$, the
action of $S_0$ is necessarily trivial on $\p[e_j+e_l]$.
Moreover, if  $X\in \g^{e_j+e_l}$, then $[K^i,X\pm \theta X]=0$, for all $i\not=j,l$, implying that       
 $[Z_0,X\pm \theta X]={1\over 2}[ K^j+ K^l ,X\pm \theta X].$ Note that  $[K^l,Z^\mu]\in\g^\lambda$ and
$[K^j,Z^\mu]\in\g^{\theta \lambda}$, where 
 $\lambda\in\Delta$ is  the root with real part $e_j-e_l$ and the same imaginary part as $\mu$.
Then, by  equating terms in the same root spaces in (\ref{AA}), one obtains  the  relations which guarantee that  $Y=[Z_0,X]\in\g^{e_j-e_l} \oplus \g^{\theta(e_j-e_l)}$. Namely
$$[K^l,Z^\mu]=-[K^j,\theta Z^\mu] \in \g^\lambda\, \qquad [K^l,\overline Z^\mu]=-[K^j,\theta \overline Z^\mu] \in\g^{\bar\lambda}\,,$$
$$[K^j,Z^\mu]=-[K^l,\theta Z^\mu] \in \g^{\theta\lambda}\, \qquad [K^j,\overline Z^\mu]=-[K^l,\theta \overline Z^\mu] \in \g^{\theta\bar\lambda}\,.$$

\sn
(c) 
If $X\in \g^{e_j}$, then  $[K^l,X\pm \theta X]=0$, for all $l\not=j$, implying that       
 $[Z_0,X\pm \theta X]=[{1\over 2}K^j +S_0,X\pm \theta X].$ 
From (\ref{AA}) one obtains  
$$\textstyle {1\over 2} [K^j,X]+[S_0,X]=-{1\over 2} [K^j,\theta X]-[S_0,\theta X],$$
and,  by equating terms in the same root spaces, one obtains 
the relations 
$$\textstyle [S_0,Z^\mu]=-{1\over 2} [K^j,\theta Z^\mu]\in\g^\mu  \qquad [S_0,\theta  Z^\mu]=-{1\over 2} [K^j, Z^\mu]\in\g^{\theta\mu}$$
 $$\textstyle [S_0,\overline Z^\mu]=-{1\over 2} [K^j, \theta \overline Z^\mu ]\in\g^{\bar\mu}  \qquad  [S_0,\theta\overline Z^\mu]=-{1\over 2} [K^j,   \overline Z^\mu ]\in\g^{\theta\bar\mu}\,,$$
which imply
  $$[Z_0,X-\theta X]= 2(-[S_0,\theta  Z^\mu]-[S_0,\theta\overline Z^\mu]+[S_0,Z^\mu]+[S_0,\overline Z^\mu]).$$
As $\mu(S_0)=:i\mu_0\in i\R$,  
the above expression becomes  
  $$2\mu_0 i (Z^\mu-\overline Z^\mu-\theta (Z^\mu-\overline Z^\mu)). $$
From $I_0^2=-Id$, one obtains 
$\mu_0=\pm {1\over 2}$.  
Depending on the value  $\mu_0$,  the pairs of roots $\mu,~\bar\mu$ can be relabelled so that $I_0P$, has the desired
expression.
\end{proof}

\bn
\begin{remark} 
\label{BASIS2} In view of Lemma \ref{COMPLEXSTRUCTURE}, one can choose a $I_0$-stable basis of $\p$, compatible with the decomposition $($\ref{DECO}$)$. 

\sn
$ (a)$ A  basis of   $\a\oplus\bigoplus_j \p[2e_j]$:  take the elements 
$A_j, P_j=-I_0A_j $, for $j=1,\ldots,r$, normalized as in  $($\ref{NORMALIZ1}$)$ and $($\ref{KJPJ}$)$;

\sn
$ (b)$ A basis of  $\p[e_j+e_l]\oplus\,\p[e_j-e_l]$:  
take 4-tuples of elements  $P,\, P',\, I_0P,\, I_0P'$, parametrized by the pairs of roots $\mu\not=\bar\mu \in\Delta^+$ satisfying   $Re(\mu)=e_j+e_l $ $($with no repetition$)$.
More precisely,
 $P=X-\theta X $ and  $P'=X'-\theta X'$, where   $X=Z^\mu+\overline Z^\mu$, $X'=i(Z^\mu-\overline Z^\mu)$, and  $Z^\mu $ is a root vector in $\g^\mu$.  
 \pn
For $\mu= \bar \mu$, we may assume $Z^\mu=\overline Z^\mu$. Then take the pair $P,\, I_0 P$.  \pn

\sn
$ (c)$ A basis of $\p[e_j]$ $($non-tube case$)$: take pairs of elements $P,\, I_0P$, parametrized by the pairs of roots $\mu\not=\bar\mu \in\Delta^+$ satisfying  $Re(\mu)=e_j $ $($with no repetition$)$.
More precisely,  $P=X -\theta X,$ where $X=Z^\mu+\overline Z^\mu$, and $Z^\mu$ is a root vector in~$\g^\mu$.
\end{remark}

\bn
\begin{lem} 
\label{BRACKETS}
 Let  $\mu\in \Delta^+$ be a root satisfying $Re(\mu)=e_j+e_l$ and  let $Z^\mu$ a root vector in
 $\g^\mu$. 
 Let  $X=Z^\mu+\overline Z^\mu  \in \g^{e_j+e_l}$ and $Y=[Z_0,X]$.
 Then
 
\mn
$(a)$ $[Y,X]+\theta[Y,X]= r_1K^j +s_1K^l$, for $r_1,\, s_1 \in \R $; 

\sn
$(b)$ $[Y,\theta X]+\theta [Y,\theta X]=r_2K^j +s_2K^l$, for $r_2,\, s_2 \in \R $.

\mn
If $\ \overline \mu \not= \mu$, let  $X'=i(Z^\mu-\overline Z^\mu)$ and $Y'=[Z_0,X']$. Then 

\mn
$(c)$ $[Y',  X]+\theta [Y',  X]=[Y',\theta X]+\theta [Y',\theta X]=0$. 

\mn
Let  $\mu$ be a root in $\Delta^+$, with $Re(\mu)=e_j$ $($non-tube case$)$ and let  $Z^\mu$  be a root vector in $\g^\mu$.
Let  $X=Z^\mu+\overline Z^\mu $ and $Y=[Z_0,X]=i(Z^\mu-\overline Z^\mu)$. Then 

\mn
 $(d)$ $[Y,X]+\theta[Y,X]=tK^j$, for $t \in\R$,
 
 \sn
 $(e)$ $[Y,\theta X]+\theta[Y,\theta X]\in\m$.
\end{lem}

\begin{proof} (a)  
By Lemma \ref{COMPLEXSTRUCTURE} (b), one has $Y={1\over 2}[K^j+K^l,X]\in\g^{e_j-e_l}
\oplus \g^{\theta(e_j-e_l)}
$.  Moreover 
 $[Y,X]={1\over 2}[[\theta E^j,X], X]+{1\over 2}[[\theta E^l,X], X]\in \g^{2e_l}
\oplus \g^{ 2e_j}.$ 
Since  such roots spaces are 1-dimensional, then    
 $$[Y,X]+\theta[Y,X]=r_1K^j +s_1K^l,\quad \hbox{ for some $r_1~s_1\in\R$}.$$
 \sn
(b) Similarly,
 $[Y,\theta X]={1\over 2}[[\theta E^j,X], \theta X]+{1\over 2}[[\theta E^l,X], \theta X]\in \g^{-2e_j}
\oplus \g^{-2e_l}$, and
 $$[Y,\theta X]+\theta [Y,\theta X]=s_1K^j+s_2K^l,\quad \hbox{for some $s_1,s_2\in\R$}.$$
%

\sn
(c) One has  
$$[Y',X]=[[Z_0,iZ^\mu-i\overline{Z^\mu}],Z^\mu+\overline{Z^\mu}]$$
$$=i [[Z_0,Z^\mu],Z^\mu]+i[[Z_0,Z^\mu],\overline{Z^\mu}]-i [[Z_0,\overline{Z^\mu}],Z^\mu] -i [[Z_0,\overline{Z^\mu}],\overline{Z^\mu}].$$
The first and the fourth terms of the above expression sum up to zero: if $\bar \mu\not= \mu$, they are both zero because otherwise there would exist a root in $\Delta^+$ with real part equal to $2e_j$ and non-zero imaginary part; if $\bar\mu=\mu$, such terms are opposite to each other. The second and the third term sum up to zero by the Jacobi identity and the fact that $[Z_0,[Z^\mu,\overline{Z^\mu}]]=0$.
One has
$$[Y',\theta X]=[[Z_0,iZ^\mu-i\overline{Z^\mu}],\theta Z^\mu+\theta \overline{Z^\mu}]$$
$$=i [[Z_0,Z^\mu], \theta Z^\mu]+i[[Z_0,Z^\mu],\theta\overline{Z^\mu}]-i [[Z_0,\overline{Z^\mu}],\theta Z^\mu]   -i [[Z_0, \overline{Z^\mu}],\theta\overline{Z^\mu}].$$
The above expression is automatically zero, if $\bar\mu=\mu$. So let's assume that $\bar\mu\not =\mu$.
Arguing as in the previous case, one has that the sum of the first and the fourth terms  is equal to zero. The second and the third terms sum up to
$2{\rm Im}([[Z_0,Z^\mu],\theta\overline{Z^\mu}])$.
Then 
$$[Y',\theta X]+\theta [Y',\theta X]=2{\rm Im}([[Z_0,Z^\mu],\theta\overline{Z^\mu}]+\theta[[Z_0,Z^\mu],\theta\overline{Z^\mu}] )=$$
\begin{equation}\label{VV}=2{\rm Im}(  [[Z_0,Z^\mu],\theta\overline{Z^\mu}]+[[Z_0,\theta Z^\mu], \overline{Z^\mu}] ).\end{equation}    
By the Jacobi identity
$$[[Z_0,\theta Z^\mu], \overline{Z^\mu}]=[[Z_0,\overline{Z^\mu}],  \theta Z^\mu]+[Z_0,[\theta Z^\mu, \overline{Z^\mu}]].$$
Observe that $[Z^\mu,\theta \overline{Z^\mu}]\in \a \oplus i\s$ and therefore 
 $[Z_0,[\theta Z^\mu, \overline{Z^\mu}]]\in\p$.
It follows that  the expression in (\ref{VV}) reduces to
$$2{\rm Im}(  [[Z_0,Z^\mu],\theta\overline{Z^\mu}]+ [[Z_0,\overline{Z^\mu}],  \theta Z^\mu])=0,$$
as desired.

\sn
(d) Since  $[Y,  X]\in\g^{2e_j}$
$$[Y,  X]+\theta[Y,X]
 =tK^j, \quad \hbox{ for some $t\in\R$} .$$
 (e) Since $[Y,  X]\in\g^{0}$,
 $$[Y,  X]+\theta[Y,X]\in\m.$$
 \end{proof}

Next, we recall a fact which will be used 
 to compute  the Levi form of an arbitrary smooth $\,K$-invariant function on $G/K$.
 For $\,X \in \k$, denote by 
$\,\widetilde X\,$ the vector field  induced on $\,G/K\,$  by the $\,K$-action 
  \begin{equation}\label{TILDEFIELDS}\,
  \textstyle\widetilde X_z:=\dds \exp sX \cdot z, \quad z\in G/K.\end{equation}
Given   a smooth $\,K$-invariant function $\,f\colon G/K \to \R\,$, 
set 
$\,d^c_{I_0}\rho:=d\rho \circ {I_0}$, so that 
$\,2i \partial  \bar \partial_{I_0}f=-dd^c_{I_0}f$. Moreover, for $\,X\in  \k$, define  $\,\mu^X: G/K \to \R\,$ by 
 $\,\mu^X(z):=d^c_{I_0}f(\widetilde X_z)\,.$ 
Then 
\begin{equation}
\label{DIFFER} 
\,d\mu^X= -\iota_{\widetilde X}dd^c_{I_0}f\,.  
\end{equation} 

The above identity  was proved in  \cite{HeSc07}, Lemma 7.1. Indeed their argument needs
 not the  plurisubharmonicity of $\,f$. 

If   the function   $\,f\,$ is  
 strictly  plurisubharmonic,  then $\,-dd^c_{I_0}f\,$ is 
a $\,K$-invariant K\"ahler form   and the map 
$\,\mu:G/K \to \k^*$, defined by 
$$ \mu(z)(X)=d^c_{I_0}f(\widetilde X_z)\,,$$
for $\,X\in \k$,  is a moment map. It  is referred to as the moment map associated
with~$\,f$.

\bn

We conclude the preliminaries with a lemma needed
in the next section.
Let $\Delta$ be the unit disc in $\C$.
Consider the ($T \ltimes \mathcal S_2$)-action on $\Delta^2$,
where $T= (S^1)^2$ acts by rotations and $\mathcal S_2$ is the group of coordinate permutations.
Let $W_{\R^2}=(\Z_2) \ltimes \mathcal S_2$ be the group acting on $\R^2$ by signed permutations.

\bigskip
\begin{lem} 
\label{CONVESS}
Let  $f:\Delta^2 \to \R$  be a smooth
$T \ltimes \mathcal S_2$-invariant strictly plurisubharmonic  function and
  let $r$, $s$ be real numbers.
Consider the $W_{\R^2}$-invariant function  $\tilde f: \R^2 \to \R$ given by $\tilde f(a_1,a_2)= f(\tanh a_1, \tanh a_2)$
and define
 $$\textstyle G_{\tilde f}(a_1,a_2):= \frac{r \sinh (2a_1) \frac{\partial \widetilde f}
{  \partial a_1}(x, y) - s\sinh (2a_2)\frac{\partial \widetilde f}
{  \partial a_2}(a_1,a_2)}{\sinh^2 a_1 -\sinh^2a_2}\,.$$

Then 
\begin{itemize}
\item [(i)] $\frac{\partial \widetilde f}
{  \partial a_1}(a_1,a_2) >0$, for every $a_1>0$, and $\frac{\partial \widetilde f}
{  \partial a_1}(a_1,a_2) <0$, for every $a_1<0$.
In particular $\frac{\partial \widetilde f}
{  \partial a_1}(0, a_2) =0$, for every  real $a_2$.
\item [(ii)] $\frac{\partial \widetilde f}
{  \partial a_2}(a_1,a_2)= \frac{\partial \widetilde f}
{  \partial a_1}(a_2,a_1).$
In particular  $\frac{\partial \widetilde f}
{  \partial a_2}(a_1, 0) =0$, for every real $a_1$.
\item [(iii)] If  $G_{\tilde f}$ extends continuously on $\,\R^2\,$ to a strictly positive function,
then $r= s>0$. In particular $G_{\tilde f}(a_1,a_2)$ is  $W_{\R^2}$-invariant as well.

\end{itemize}
\end{lem}

\medskip
\begin{proof}
(i) 
For $\,a_1>  0$ let $s_1 \in (-\infty,0)$ be such that 
$\tanh a_1=e^{s_1}$. Since $f$ is $T$-invariant and strictly plurisubharmonic,
the function $\,s_1 \to f(e^{s_1},a_2)$ is strictly convex. Moreover, the limit
$\lim_{s_1 \to  -\infty}  f(e^{s_1},a_2)= f(0,a_2)\, $
is finite. Hence the  function is strictly increasing and its derivative
$ e^{s_1} \frac{\partial f}{\partial a_1}(e^{s_1},a_2)$ is strictly positive. 
As
$$\textstyle \frac{\partial \widetilde f}
{  \partial a_1}(a_1,a_2) =\frac{1}{\cosh^2 a_1}\frac{\partial f}{\partial x}(\tanh a_1,a_2)\,,$$
and $\tanh a_1 =e^{s_1}$, statement (i) follows.

(ii) The $\mathcal S_2$-invariance of $\tilde f$ implies that
$$\textstyle \lim_{\varepsilon \to 0} \frac{{\tilde f}(a_1, a_2+ \varepsilon)}{\varepsilon}= 
\lim_{\varepsilon \to 0}  \frac{ {\tilde f}(a_2+ \varepsilon, a_1)}{\varepsilon}\,,$$ 
and  (ii) follows.

(iii)  Let $a_1>0$. From (ii) it follows that 
$G_{\tilde f}(a_1,0)= \frac{r \sinh (2a_1) \frac{\partial \widetilde f}
{  \partial a_1}(a_1,0)} {\sinh^2 a_1}\,.$
Since such quantity is assumed to be strictly positive and 
$\frac{\partial \widetilde f}
{  \partial a_1}(a_1, 0)>0$, it follows that $r>0$.
By choosing $a_1=0$ and $a_2>0$,  one obtains that $s>0$.

Next we show that $r=s$. For $a_1>a_2>0$ one has $\sinh^2 a_1 -\sinh^2a_2>0$. Then the positivity
of $G_{\tilde f}(a_1,a_2)$ implies that
$$
\textstyle \frac{r}{s}> \frac{\sinh (2a_2)\frac{\partial \widetilde f}
{  \partial a_2}(a_1, a_2)}{\sinh (2a_1) \frac{\partial \widetilde f}
{  \partial a_1}(a_1, a_2)}\,.$$
Consequently, for $a_1$ converging to  a fixed $a_2>0$, statements (i) and (ii) imply 
$\frac{r}{s} \geq 1$.
An analogous argument, with $0<a_1<a_2$, implies $\frac{r}{s} \leq 1$. 
As a consequence, $\frac{r}{s}=1$.
\end{proof}


\bigskip
\section{The Levi form of a $K$-invariant function}
\label{REALLEVI}
\bigskip
Let $\,G/K\,$ an
 irreducible non-compact Hermitian symmetric space  of rank $\,r$.
From the  decomposition $\,G=K \exp \a\, K\,$    
one obtains a bijective correspondence between $\,W$-invariant domains in $\,\a\,$  and $\,K$-invariant domains in 
$\,G/K\,$,  namely
\begin{equation}\label{DA}\,\mathcal D_{\a} \to D:=K \exp \mathcal D_{\a}K/K  \,.\end{equation}
In addition, every $K$-invariant function $f:D \to \R$ is uniquely determined by the $W$-invariant function 
 $\tilde f:\mathcal D_{\a} \to \R$, given by 
 \begin{equation} \label{EFFETILDE}\tilde f(H)=f(\exp (H)K).\end{equation}

\sn
The goal of this section is express the Hermitian form
$h_f(\,\cdot\,,\,\cdot\,):= -dd^c f(\,\cdot\,,I_0\,\cdot\,)$ of a smooth $\,K$-invariant function $f$ on a $K$-invariant domain $D\subset G/K$ in terms of the   first and second derivatives of the corresponding $\tilde f$ on  $\,\mathcal D_{\a}$. This will enable us to characterize smooth $K$-invariant strictly plurisubharmonic functions on a Stein $K$-invariant domain $D$ in $G/K$ by an appropriate 
differential positivity condition on the corresponding functions on $\mathcal D_\a$ (see Thm.$\,$\ref{BIJECTIVEK1} 
and Cor.$\,$\ref{PSHONR}). As $f$ is 
$K$-invariant, also $-dd^cf(\,\cdot\,,I_0\,\cdot\,)$ is  $K$-invariant. Therefore it will be sufficient
to carry out the computation along the slice $\exp \a K$, which meets all $K$-orbits.


  


For   $\,z=aK \,$, with  $\,a=\exp(H)$ and $H\in\a$,  one has \begin{equation}\label{TILDEVSHAT}\widetilde X_z = a_*F_aX ,\qquad 
a_*X=\widetilde{F_a^{-1} X}_z, \end{equation}
where $F_a\colon \p\to\p$ is the map given by $F_a := \pi_\# \circ \Ad_{a^{-1}}|_{\p }\,,$
and $\, \pi_\# :\g  \to \p \,$ is the linear projection  along $\,\k $.
In particular one can  verify that 
\begin{equation}\label{TILDEFIELDS2}
 \widetilde K_z=-a_*\sinh\alpha(H)P,\end{equation}
for $P=X^\alpha-\theta X^\alpha\in \p[\alpha]$ and $K=X^\alpha+\theta X^\alpha\in\k[\alpha]$,
with $\alpha \in \Sigma^+$.

Denote by $a_1,\ldots,a_r$  the coordinates  induced on $\a$ by the  basis $A_1,\ldots,A_r$ of~$\a$  (cf. 
Rem.  \ref{BASIS2}(a)). 

\bigskip
\begin{prop} 
\label{LEVI} 
 Let $D\subset G/K$ be a $K$-invariant domain. Let $\,f:D \to \R\,$ be a smooth $K$-invariant function. Fix $\,a=\exp H$, with $H=\sum_ja_jA_j\in \mathcal D_{\a}$. Then, in the basis  of $\,\p\,$ defined in  Remark \ref{BASIS2}, the   Hermitian form $h_f$  at $z=aK\in D$ is as follows.\ 
\begin{itemize}
\smallskip
\item[(i)]  The spaces $\,a_*\a$, $\,a_*I_0\a$, $\,a_* \p[e_j+ e_l]\,$, $\,a_* \p[e_j- e_l]\,$ 
and $a_* \p[e_j]\,$ are pairwise $\,h_f$-orthogonal.
  \end{itemize}
As the form $h_f$ is  $I_0$-invariant,  it is determined by its restrictions to  the blocks $a_*\a$, $a_*\p[e_j+e_l]$ and $a_*\p[e_j]$.  
The non-zero entries of   
$h_f$ on each of these blocks are  given as  follows.
\begin{itemize}

\smallskip
\item[(ii)]  For $A_j,A_l\in\a$ 
one has

\noindent
$\,h_f(a_*A_j,a_*A_l)=
2\coth (2a_j)\frac{\partial \widetilde f}
{  \partial a_j}(H) \delta_{jl}+\frac{\partial^2 \widetilde f}
{  \partial a_j   \partial a_l}(H)\,$;

\smallskip
\item[(iii)]   For 
$P,\, P'$ as in Remark \ref{BASIS2}(b) 
one has

\smallskip
\noindent
$$\,h_f(a_*P ,a_*P )=h_f(a_*P' ,a_*P' )=
$$
$$ \textstyle =\frac{B(P,P)}{b}\frac{1}{\sinh (a_j+a_l)\sinh  ( a_j-a_l)}
\big ( \sinh (2a_j)\frac{\partial \widetilde f}{  \partial a_j}(H)-
   \sinh (2a_l)\frac{\partial \widetilde f}{  \partial a_l}(H)\big )\,,$$ 
   where 
$b:=B(A_1,A_1)= \dots =B(A_r,A_r)$.
   In particular, with respect to the  basis of $a_* \p[e_j+e_l]$ defined in 
Rem.  \ref{BASIS2}$($b$)$,
   the form $h_f$ is diagonal.
  

\smallskip
\item[(iv)] $($non-tube case$)$ For $P\in \p[e_j]$,  as in Remark  \ref{BASIS2} (c), one has

\noindent
$$\,h_f(a_*P ,a_*P )=  \textstyle \frac{B(P,P)}{b}\coth(a_j)
\frac{\partial \widetilde f}
{  \partial a_j}(H)\, .$$ 
In particular, with respect to the  basis of $a_* \p[e_j]$ defined in 
Rem.  \ref{BASIS2}$($c$)$,
   the form $h_f$ is diagonal.

  \end{itemize}
\end{prop}

\medskip
\begin{proof} In order to exploit the relation (\ref{DIFFER}), 
 we first compute  $d^cf(\widetilde X_{z})$, for $X\in \k$ and
$z \in G/K$.
By the $\,K$-invariance of $\,f\,$ and of $\,I_0$ one has 
\begin{equation}\label{DCf} \, d^cf(\widetilde X_{k\cdot z})=d^cf( \widetilde {\Ad_{k^{-1}}X}_{z})\,,\end{equation}

\nsmallskip
for every $\,z \in G/K\,$ and $\,k \in K$. Thus it is sufficient to carry out the
computation for $z=aK$.
We first  assume that $\alpha(H) \not=0$  for all $ \alpha\in
\Sigma$,  and later    obtain the desired result by passing to the limit for $H$
approaching the hyperplanes defined by $\{\alpha=0\}$.

\medskip
Recall the decomposition of  $\k$  given  in (\ref{DECO}). 
For all  $M \in \m$, 
one has $\widetilde M_z=0$, and therefore 
\begin{equation}
\label{MENOUNO}
d^cf (\widetilde M_z)=0\,.
\end{equation}
 
\sn
For $\,K = X^\alpha+ \theta X^\alpha$ in $\,\k[\alpha]$, with $\,\alpha \not= 2e_1, \dots, 2e_r$,   set
$\,P = X^\alpha- \theta X^\alpha\,$ in $\,\p[\alpha]\,$. Then 
$\,  I_0P = Y^\beta- \theta Y^\beta$ in $\p[\beta]$, for some $\beta\not= 0$.  Set 
$\,C= Y^\beta+ \theta Y^\beta$ in $\k[\beta]$.
Then, by (\ref{TILDEFIELDS2}) and the $K$-invariance of $\,f\,$, one has

$$d^c f( \widetilde {K}_z )=
-d f (I_0a_* \sinh \alpha(H) P )=-d f(a_* \sinh \alpha(H) I_0P)  $$
\begin{equation}
\label{ZERO}
=d f \big ( \textstyle \frac {\sinh \alpha (H)}{ \sinh \beta(H)} \widetilde{C}_z \big )= 0.
\end{equation}

\mn
Finally, in the case of  $K^j\in \k[2e_j]$, for $\,j=1, \dots\, r$  (cf. (\ref{KJPJ})),  one has

$$d^c f  ( \widetilde{ K^j}_z)=
-d f \big (\textstyle  I_0a_* \sinh (2a_j)
P^j \big ) =-d f\big(
a_*\sinh (2a_j)A_j \big)=$$
\begin{equation}
\label{UNO}
=\textstyle- \dds f \big (\exp (H+ \sinh (2a_j)sA_j)K \big )=  -\sinh(2a_j) \frac{\partial \widetilde f}
{  \partial a_j}(H)\,.
\end{equation}

\bn
Next  we prove statement {\bf (i)}. 
Let $\alpha$ and  $\gamma$ be distinct roots with $\alpha \in  \Sigma^+$ 
and $\gamma \in \{0\} \cup (\Sigma^+ \setminus \{2e_1,\dots,2e_r\})$, with the convention $\p[0]:=\a$. 
Let $P \in \p[\alpha]$ and $Q\in  \p[\gamma]$. Write
$P =X^\alpha-\theta X^\alpha$, with $X^\alpha\in\g^\alpha$, and $I_0Q =Y^\beta-\theta Y^\beta$, with  $Y^\beta\in\g^\beta$, for some  $\beta \in \Sigma^+$. Then
$\,  a_*P =\textstyle -\frac{1}{\sinh \alpha(H)} \widetilde K_z \,$, with   $K =X^\alpha+\theta X^\alpha  \in \k[\alpha]$,  and
$\,a_*I_0Q  =\textstyle-\frac{1}{\sinh \beta(H)}  \widetilde C_k \,,$
with   $C =Y^\beta+\theta Y^\beta \in \k[\beta]$. 
Therefore
$$h_f(a_*P ,a_*Q)=-dd^c f(a_*P ,\,a_*I_0Q)=\textstyle
\frac{1}{\sinh \alpha(H) \sinh \beta(H)}
\ddt \mu^{K}( \exp tC  \cdot z)=$$
$$\textstyle =\frac{1}{\sinh \alpha(H) \sinh \beta(H)}
\ddt d^c f( \widetilde K_{\exp tC \cdot z})$$
which, by (\ref{DCf}), becomes  
$$\textstyle \frac{1}{\sinh \alpha(H) \sinh \beta(H)}
\ddt d^c f( \widetilde {\Ad_{\exp -tC }K} _{z})=$$
$$\textstyle =\frac{1}{\sinh \alpha(H) \sinh \beta(H)}
\ddt d^c f \big ( \widetilde{K}_{z} -t 
\widetilde{[C ,K ]}_{z} +o(t^2) \big ) =$$
$$\textstyle =-\frac{1}{\sinh \alpha(H) \sinh \beta(H)}
d^c f( \widetilde {[C ,K ]} _{z})\,.$$
Hence 
\begin{equation}
\label{FORMULONE}
\textstyle h_f(a_*P ,a_*Q)=
-\frac{1}{\sinh \alpha(H) \sinh \beta(H)}
d^c f( \widetilde {[C ,K ]} _{z})\,.
\end{equation}
The brackets
$$[C ,K ]= ([Y^\beta,X^\alpha]+\theta [Y^\beta,X^\alpha])+
([Y^\beta,\theta X^\alpha]+\theta [Y^\beta,\theta X^\alpha])\,,$$
lie in  $\,\k[\alpha+\beta] + \k[\alpha-\beta]$.
Since $\alpha\in\Sigma^+$ and $\gamma\in  \{0\} \cup (\Sigma^+ \setminus \{2e_1,\dots,2e_r\})$ are distinct, the spaces $\k[\alpha+\beta]$ and $\k[\alpha-\beta]$ have zero intersection with $\oplus_j\k[2e_j]$. Then the expression (\ref{FORMULONE}) vanishes by   (\ref{MENOUNO}) and   (\ref{ZERO}), and   
the spaces
  $a_*\p[\alpha]$ and $ a_*\p[\gamma]$ are $h_f$-orthogonal, as claimed. By the $I_0$ invariance of $h_f$, this also implies that $a_*\a$ is $h_f$-orthogonal
to $\oplus_{\alpha \in \Sigma^+}a_*\p[\alpha]$.  This  concludes the proof of~(i).

\sn
Next we examine the Hermitian tensor $h_f$ on the various  blocks.

\sn
 {\bf \boldmath (ii)  The   Hermitian form $h_f$ on $a_*\a$. }
 
 \sn
Let $A_j,A_l\in\a$.  Since  $I_0A_l=-P^l$, one has
 $$h_f(a_*A_j,a_*A_l)=  
 -dd^cf(a_*P^l ,\,a_*A_j)=
\textstyle {1\over \sinh (2a_l)}dd^c f
((\widetilde{K^l })_z,(\widetilde{A_j})_z)=$$
$$\textstyle
 - {1\over \sinh (2a_l)} \ddt \mu^{K^l }(\exp tA_j \cdot z)=
 \textstyle -{1\over \sinh (2a_l)} \ddt d^cf(  \widetilde K^l_{\exp tA_j \, z}) $$
 $$ =\textstyle {1\over \sinh (2a_j)} \ddt df( I_0 ({\exp (H+tA_j) })_* \sinh 2e_j(H+tA_j))\widetilde P^l) 
=$$
$$
\textstyle 
{1\over \sinh (2a_l)}
\ddt \sinh 2e_j(H+tA_j) \frac{\partial \widetilde f}
{  \partial a_l}(H+tA_j) 
=$$ 
$$\textstyle {1\over \sinh (2a_l)} \big (
2 \cosh(2a_l)\frac{\partial \widetilde f}
{  \partial a_l}(H) \delta_{j,l}+ \small{ \sinh (2a_l)}\frac{
\partial^2 \widetilde f}
{  \partial a_j  \partial a_l}(H) \big)\,.$$

\sn
The above expression  is well defined also for those $H=\sum_ja_jA_j$ with some zero coordinate.
Assume for example $a_l=0$. 
By the $\,W$-invariance, $\widetilde f$ 
is even with respect to the coordinate $a_l$ and consequently its derivative $\frac{\partial \widetilde f}
{\partial a_l}$
vanishes for  $a_l=0$. Therefore
$$
 \textstyle \lim_{a_l\to 0}= 2 \coth( 2a_l)
 \frac{\partial \widetilde f}
{  \partial a_l}=  \frac{\partial^2 \widetilde f}
{  \partial\, a_l^2}, 
$$
and  the above quantity smoothly extends to the hyperplane 
$a_l=0$, for all $l=1,\ldots,r$.   This concludes the proof of (ii).

\bn
 {\bf \boldmath (iii)   The   Hermitian form $h_f$ on $a_*\p[e_j+ e_l]$. }

\sn
Let $P,\, Q\in\p[e_j+e_l]$ be elements of the basis of Remark \ref{BASIS2}\,(b), arising from   roots  $\mu,\, \nu\in \Delta^+$, respectively, with $\nu\not=\mu,\,\bar \mu$. 
Then 
 $h_f(a_*P,a_*Q)  =0$, due the fact that for such roots 
 $ [Z^\mu\pm \overline Z^\mu,Z^\nu\pm \overline Z^\nu]=0,$ for all $Z^\mu\in\g^\mu$ and $Z^\nu\in\g^\nu$.

Next, let $P,\,P'\in \p[e_j+e_l] $ be elements of the basis  given in  Remark \ref{BASIS2}\,(b), arising from  the same  root   $\mu\in \Delta^+$. 
Write $P=X-\theta X$,  with $X=Z^\mu+\overline{Z^\mu}$, and $I_0P=Y-\theta Y$, with $Y=[Z_0,X]$. Likewise write   $P'=X'-\theta X'$, with $X'=iZ^\mu-i\overline{Z^\mu}$,  and $I_0P'=Y'-\theta Y'$, with   $Y'=[Z_0,X']$.

From (\ref{FORMULONE}) it follows that 
 $$h_f(a_*P,a_*P )= 
\textstyle -\frac{1}{\sinh {(a_j+a_l)} \sinh {(a_j-a_l)}}
\big (d^c f(\widetilde {[ C,K]}_z)\big ),$$
where $K=X+\theta X$ and $C=Y+\theta Y$.
By Lemma \ref{BRACKETS}(a)(b),
the above  expression   equals 
$$\textstyle -\frac{1}{\sinh {(a_j+a_l)} \sinh {(a_j-a_l)}}
\big (r d^c f(\widetilde {K^j}_{z})-  sd^c f(\widetilde {K^l}_{z})\big ) 
$$
\begin{equation} \label{III} 
= \textstyle  \frac{1}{\sinh {(a_j+a_l)} \sinh {(a_j-a_l)}}
\big (r\sinh (2a_j) \frac{\partial \widetilde f}{ \partial a_j}(H)-
  s\sinh (2a_l)\frac{\partial \widetilde f}{  \partial a_l}(H)\big )\,,\quad r,s\in\R.\end{equation}

\sn
In a similar way, one  obtains 
\begin{equation} \label{IIII} h_f(a_*P',a_*P')=h_f(a_*P ,a_*P ),
\end{equation}
and, from  Lemma \ref{BRACKETS}(c), 
$$h_f(a_*P,a_*P')=0.$$
An argument similar to the one used in (ii) shows that the above expressions (\ref{III}) and (\ref{IIII}) smoothly extend  to the hyperplane $(e_j-e_l)(H)=0$.\pn

Moreover the  quantities (\ref{III}) and (\ref{IIII})  are  strictly positive for   the  strictly plurisubharmonic  potential $\rho$ of  the Killing  metric  of $G/K$  given in  Proposition \ref{POTENTIALK}. Then
  (iii) in Lemma \ref{CONVESS}  implies that $r=s>0$. Finally, 
 as $h_\rho(a_*P,a_*P)= B(P,P)$, a simple computation shows that 
 $r = B(P,P)/b$.
This concludes the proof of (iii).

\bn
(iv)  {\bf \boldmath The   Hermitian form $h_f$ on $a_*\p[e_j]$}.

\sn
Let $P,\, Q\in\p[e_j]$ be elements of the basis   given in  Remark \ref{BASIS2}\,(c), arising from   roots  $\mu,\, \nu\in \Delta^+$,  respectively, with $\nu\not=\mu,\,\bar \mu$. 
Then 
 $h_f(a_*P,a_*Q)  =0$, 
due the fact that for such roots 
 $ [Z^\mu\pm \overline Z^\mu,Z^\nu\pm \overline Z^\nu]=0,$ for all $Z^\mu\in\g^\mu$ and $Z^\nu\in\g^\nu$.
In addition, by the $I_0$-invariance of $h_f$ one has $h_f(a_*P,a_*I_0P)=0.$ 

In order to compute $h_f(a_*P,a_*P)$,  write $P=X-\theta X$ and $I_0P=Y-\theta Y$, with $X=Z^\mu+\overline{Z^\mu}$ and 
$Y=iZ^\mu-i\overline{Z^\mu}$ (Lemma \ref{COMPLEXSTRUCTURE}\,(c)).
Then, from (\ref{FORMULONE}) it follows that 
 $$h_f(a_*P,a_*P)=\textstyle -\frac{1}{\sinh^2(a_j)}
d^c f(\widetilde{[C,K]}_z),$$
for $K=X+\theta X$ and $C=Y+\theta Y$.
By Lemma \ref{BRACKETS}\,(d)(e), one obtains 
 $$\textstyle
 h_f(a_*P,a_*P)= -\frac{1}{\sinh^2(a_j)}d^c f({t\widetilde K^j}_{z}) = 2t\textstyle  \coth (a_j)
\frac{\partial \widetilde f}
{  \partial a_j}(H),\qquad t\in\R.$$

Note that  the above expression smoothly extends  on $\mathcal D_\a$, since $\frac{\partial \widetilde f}
{  \partial a_j}=0$ on the hyperplane $a_j=0$.
 Moreover, by (i) of Lemma \ref{CONVESS}, for $a_j>0$ one has $\frac{\partial \widetilde f}
{  \partial a_j}> 0$.

As in the previous case, one shows that $t=B(P,P)/b$ by
computing the above quantity for the  strictly plurisubharmonic  potential $\rho$ of  the Killing  metric  of $G/K$  given in  Proposition \ref{POTENTIALK}.  This completes the proof of  statement~(iv).
\end{proof}
\sn

 \bn
${\bf Remark.}$ The Levi form
 $L_f^\C$ of $f$   is given by
$$L_f^\C(Z,\overline W)=2(h_f(X,Y)+ih_f(X,I_0Y)),$$
where $Z=X-I_0X$ and $W=Y-I_0Y$ are elements in $(\p^\C)^{1,0}$. 
One easily sees that $L_f^\C$  is (strictly) positive definite if and only if $h_f$ is (strictly) positive definite.


\bigskip
\section{$K$-invariant psh functions vs. $W$-invariant logcvx functions}
\label{PSH}

\bigskip
Let $\,G/K\,$ be an
 irreducible non-compact Hermitian symmetric space  of rank $\,r$ and let $D\subset G/K$ be a Stein, $K$-invariant domain. 
The goal of this section is to prove a characterization of various classes of $K$-invariant plurisubharmonic functions on $D$ in terms of   appropriate  
conditions of the corresponding functions on  $\mathcal D_\a$ (see (\ref{DA})). As an application, we reproduce the characterization of Stein $\,K$-invariant domains in $\,G/K$ (Cor.\,\ref{CHARACTER}), 
  outlined in \cite{BeDa91}, Thm.\,$3^\prime$ and Thm.\,4 (see also 
\cite{FeHu93}).

In the  smooth case we prove that a smooth $K$-invariant function $f$ of $D$ is strictly plurisubharmonic if and only if the associated function $\tilde f$ (see (\ref{EFFETILDE})) satisfies a positivity condition arising from  Proposition \ref{LEVI}.

\bn 

 Denote  by $\,\Delta^r\,$ the orbit of the base point  $eK\in G/K$ under the  commuting $\,SL_2(\R)$'s generated by the triples defined in (\ref{NORMALIZ1})-(\ref{KJPJ}). It is well-known (cf. \cite{Wol72})  that  $\Delta^r$    is biholomorphic to the unit  polydisk in $\C^r$. One has $\,\Delta^r=T\exp \a K$,
 where $T\cong (S^1)^r$ is the $\,r$-dimensional  torus in $\,K\,$ whose Lie algebra is generated
 by $\,K^1,\dots ,K^r$, and
   $$ \exp (a_1,\ldots,a_r)K=(\tanh(a_1),\ldots,\tanh(a_1)),\quad\hbox{for $(a_1,\ldots,a_r)\in\a$}.$$
 The polydisk $\Delta^r$ is a  ``thick slice"  for the $K$-action in $G/K$, in the sense that 
 $\,K \cdot\Delta^r =G/K.$ 
If $D$ is a $K$-invariant domain in $G/K$, then    the  Reinhardt domain associated to $D$ is defined as
  $$\,R:=D \cap \Delta^r \quad \quad {\rm and } \quad \quad D=K\cdot R\,.$$ 
 We will show that if $D$ 
   is Stein, then $R$ is necessarily  connected.  (It should be remarked that,  despite its appellation,  a Reinhardt  ``domain"   is 
  open in $\,\C^r\,$ but need not
  be connected).

For a  Reinhardt  domain  $R$  in $\Delta^r$, define the set 
$\,\mathcal  D = \{(a_1,\ldots, a_r) \in \R^r \ : \ (\tanh a_1,\ldots,\tanh a_r) \in R\},$
with the property that the image of the map
$$\mathcal  D  \to R\, \quad \quad (a_1,\ldots, a_r) \to (\tanh a_1,\ldots,\tanh a_r) $$
coincides with $\,R \cap \R^r$. One has  $R=T\cdot (R \cap \R^r),$ with
$T\cong (S^1)^r$. 
Given  a smooth $\,T$-invariant function $\,f\,$ on  $\,R\,$
 define $\,\tilde f: \mathcal D  \to \R $ by 
$$\tilde f (a_1,\ldots, a_r) =f (\tanh a_1,\ldots,\tanh a_r)\,.$$
By the $T$-invariance of $f$,  the function $\tilde f$ is $\,(\Z^2)^r$-invariant.

Denote by  $LogConv^{\infty,+}(\mathcal D )^{(\Z_2)^r}$   the class
of smooth 
 functions on $\,\mathcal D \,$ which are even in each variable and  
 such that the form defined in (ii) of Proposition \ref{LEVI} is strictly positive definite. 
 The next proposition characterizes  $T$-invariant smooth strictly plurisubharmonic functions on  $R$  by  elements in
  $LogConv^{\infty,+}(\mathcal D )^{(\Z_2)^r}$.  It is  an intermediate step in the proof of the main  theorem in the smooth case, but it may   be of independent interest 
  in the context of Reinhardt domains.

\mn\begin{prop}
\label{PSHPOLIDISC} 
Let $\,f\,$ be a smooth $\,T$-invariant function on a Reinhardt domain
$\,R\,$
  in $\,\Delta^r$.
Then $\,f\,$ is strictly plurisubharmonic if and only if $\,\tilde f\,$
belongs to $LogConv^{\infty,+}(\mathcal D)^{(\Z_2)^r}$.
\end{prop}

\begin{proof}
In polar coordinates $\,(\rho_j, \theta_j)\,$, with 
$z_j=\rho_je^{i\theta_j}\not=0$, one has
$$\,\textstyle \partial_{z_j}=
\frac{e^{-i\theta_j}}{2\rho_j}(\rho_j\partial_{\rho_j} 
-i\partial_{\theta_j})\, \ \ \quad\ \  \, \partial_{\bar z_j}=
\frac{e^{i\theta_j}}{2\rho_j}(\rho_j\partial_{\rho_j} 
+i\partial_{\theta_j})\,.$$ 
One easily sees that, for $z_jz_l\not=0$,   
\begin{equation}\label{LEVIREIN}
 \textstyle 4\frac{\partial^2 f}{\partial \bar
z_j   \partial z_l} (z_1,\dots, z_r) 
  = \frac{1}{\rho_j}\frac{\partial f}{\partial  \rho_j  }(\rho_1,\dots,
\rho_r) \delta_{jl}
  +e^{i(\theta_j-\theta_l)}\frac{\partial^2  f}{\partial \rho_j   \partial
\rho_l}(\rho_1,\dots, \rho_r)\,.\end{equation} 
The above quantity  extends smoothly through the hyperplanes
  $z_j=0$ (and therefore to the whole
domain)  whenever $j=l$, while
 $$\textstyle 4\frac{\partial^2 f}{\partial \bar z_j   \partial z_l}
(z_1,\dots, z_r)=0, \quad \hbox{ for $j\not=l$ and $z_jz_l=0$}.$$

 \sn
For $\,\rho_1=\tanh a_1, \dots, \rho_r=\tanh a_r$, one has
\begin{equation}
\label{DEREFFETILDE}
\textstyle 
\frac{\partial \widetilde f}{\partial a_j }(a_1,\dots, a_r)=\frac{\partial
f}
{  \partial \rho_j } {\scriptstyle (\tanh a_1, \dots, \tanh a_r)}\frac{1}{\cosh^2
a_j}\end{equation}
\begin{equation} \label{HESSIANEFFETILDE}\textstyle 
\frac{\partial^2 \widetilde f}{\partial a_j   \partial
a_l}(H)=\frac{\partial^2 f}
{  \partial  \rho_j   \partial  \rho_l}{\scriptstyle (\tanh a_1, \dots, \tanh a_r)}
\frac{1}{\cosh^2 a_j \cosh^2 a_l}-\delta_{jl}
\frac{\partial f}
{  \partial  \rho_j }{\scriptstyle (\tanh a_1, \dots, \tanh a_r)} \frac{2\sinh
a_j}{\cosh^3 a_j}\,,\end{equation}
and likewise 
$$\textstyle \frac{\partial^2 \widetilde f}{\partial a_j   \partial a_l}(H)=0,\quad
\hbox{for $j\not= l$ and $a_ja_l=0$}.$$

A  simple computation combining formulas (\ref{DEREFFETILDE})  and
(\ref{HESSIANEFFETILDE})   with (\ref{LEVIREIN}),  shows that 
  $\textstyle 4
\frac{\partial^2 f}{\partial \bar z_j   \partial z_l}(z_1,\ldots,z_r)$
is given by
 $$ \begin{cases}\textstyle \cosh^4 a_j  \,
\big (2\coth (2a_j)\frac{\partial \widetilde f}
{  \partial a_j}(a_1,\dots,a_r) +\frac{\partial^2 \widetilde f}
{  \partial a_j^2  }(a_1,\dots,a_r) \big),  ~\hbox{for   $j=l,$}\\  
  \cosh^2 a_j\,e^{i\theta_j} \cosh^2 a_l\,e^{-i\theta_l}\, \frac{\partial^2 \widetilde f}
 {  \partial a_j   \partial a_l}(a_1,\dots,a_r),  ~\hbox{for $j\not=l$ and
$z_jz_l\not=0,$}\\
0,~\hbox{for $j\not=l$ and $z_jz_l=0$.} \end{cases}$$

\sn
Then,  for    $(z_1,\ldots,z_r)\in R$, one has
$$\textstyle \big (4 \frac{\partial^2 f}{\partial \bar z_j   \partial
z_l}\big )_{j,l}=C \big ( \frac{\partial^2 \tilde f}{\partial   a_j  
\partial a_l}+\delta_{jl}\, 2\coth(2a_j)\frac{\partial  \tilde f}{\partial  
a_j }  \big )_{j,l}\overline C,$$
where $C$ is the diagonal matrix with diagonal entries 
$$\textstyle c_{jj}= \begin{cases}\cosh^2(a_j)e^{i\theta_j}, ~~\hbox{for
$z_j\not=0,$}\\
\cosh^2(a_j), ~~\hbox{for $z_j=0$}.\end{cases}$$ 
It follows that  $f$  
is strictly plurisubharmonic  
 if and only if  
$\,\tilde f\,$
belongs to the class $LogConv^{\infty,+}(\mathcal D)^{(\Z_2)^r}$.
\end{proof}

\bigskip

Let $R$ be a Reinhardt domain  in $ (\Delta^*)^r$ and let 
 $$\,\mathcal D_{\log}:=\,\{(s_1,\, \dots , \,s_r) \in (\R^{<0})^r \ : \  
 (e^{s_1},\, \dots , \,e^{s_r}) \in R\,\}\,$$
 be its logarithmic image.
For a $T$-invariant function $f$ on $R$,
define $\widehat f\colon{\mathcal D_{\log}}\to\R$ by
\begin{equation}\label{EFFEHAT}\widehat f(s_1,\ldots,s_r):=f(e^{s_1}, \ldots,e^{s_r}).
\end{equation}
It is well known that  if $f$ is smooth, then it is strictly plurisubharmonic if and only if $\widehat f$ has strictly positive definite Hessian. 
The next remarks elucidate
 the significance  of the class   $ LogConv^{\infty,+}(\mathcal D)^{(\Z_2)^r}$.

\medskip
\begin{remark} 
\label{LOGCONV}
Let $R $ be a Reinhardt domain in $ (\Delta^*)^r$ and let $f$ be a smooth $T$-invariant function on $R$. Then 
  $\tilde f$ belongs to   $LogConv^{\infty,+}(\mathcal D)^{(\Z_2)^r}$ if and only if 
  the smooth function
$\hat f$ has everywhere strictly positive Hessian. 
\end{remark} 
\begin{proof} One has 
\begin{equation}\label{DEREFFEHAT}\textstyle 
\frac{\partial \hat f}{\partial s_j }(s_1,\dots,s_r)=\frac{\partial f}
{  \partial \rho_j }(e^{s_1},\dots, e^{s_r})e^{s_j}\,,\end{equation}
\begin{equation}\label{HESSIANEFFEHAT}\textstyle 
\frac{\partial^2 \hat f}{\partial s_j   \partial s_l}(s_1,\dots,s_r)=\frac{\partial^2 f}
{  \partial \rho_j   \partial \rho_l}(e^{s_1},\dots, e^{s_r})e^{s_j}e^{s_l} +\delta_{jl}
\frac{\partial f}
{  \partial \rho_j }(e^{s_1},\dots, e^{s_r}) e^{s_j}\,.\end{equation}
Then,  by letting $\,e^{s_1}=\tanh a_1, \dots, e^{s_r}=\tanh a_r$, with $a_1,\ldots,a_r>0$,   
and  combining formulas (\ref{DEREFFEHAT}) and  (\ref{HESSIANEFFEHAT}) with (\ref{DEREFFETILDE}) and (\ref{HESSIANEFFETILDE}), one obtains
$$\textstyle 2\coth (2a_j)\frac{\partial \widetilde f}
{  \partial a_j}(H) \delta_{jl}+\frac{\partial^2 \widetilde f}
{  \partial a_j   \partial a_l}(H)\,=\textstyle \frac{4}{\sinh 2a_j \sinh 2a_l}
\frac{\partial^2 \hat f}{\partial s_j   \partial s_l}(s_1,\dots,s_r)\,.$$
Hence  $\tilde f \in LogConv^{\infty,+}(\mathcal D)^{(\Z_2)^r}$ 
 if and only if  
$\,\hat f \,$  has everywhere strictly positive Hessian.
\end{proof}

\medskip
\begin{remark} 
Let $R$ be an  arbitrary  Reinhardt domain and let  $f|_{R \cap(\Delta^*)^r}$ denote the restriction of $f$ to $R \cap(\Delta^*)^r$. 
The strict positivity of the Hessian  of  $ f|_{R \cap(\Delta^*)^r}$ on $R \cap(\Delta^*)^r$ does not imply the strict plurisubharmonicity of $f$ on the coordinate hyperplanes (and therefore on the whole $R$). 
For instance, despite the fact that it has strictly positive   Hessian  on  $R \cap(\Delta^*)^r$, the function $g(z)=|z|^4$ is not plurisubharmonic at $z=0$. In contrast, this fact is detected by  the vanishing of the form 
$$\textstyle   \frac{\partial^2 \widetilde g}
{  \partial a   \partial a}+2\coth (2a)\frac{\partial \widetilde g}
{  \partial a}= 
16\frac{\tanh^2 a}{\cosh^4 a}\,,$$
at $a=0$, which shows that the associated function $\tilde g(a)= \tanh(a)^4$  does not belong to  $LogConv^{\infty,+}(\R)^{(\Z_2)}$.
\end{remark}

\bigskip
Let $R\subset \Delta^r$ be a Reinhardt domain associated to a $K$-invariant domain in $G/K$. 
In this case, $R$   is also invariant under coordinate permutations, which arise from the Weyl group action on $\a$. If such  a   
Reinhardt domain is Stein, then
there are two possibilities:
\begin{itemize}
\item [(a)] $\,R\,$ intersects the coordinate hyperplanes. 
Then it is complete (cf. \cite{Car73}, Thm. 2.12).
In particular it contains the origin and is connected.
\item [(b)] $\,R\,$ does not intersects the coordinate hyperplanes,   i.e.
$\,R\subset (\Delta^*)^r\,$. Then $R$ is logarithmically convex.
\end{itemize}

\mn 

The next  proposition  shows that  a Stein Reinhardt domain $R$ associated to a Stein $K$-invariant domain $D\subset G/K$  is necessarily connected (even when $0\not\in R$), a fact already pointed out  in \cite{BeDa91},  Thm.\,$3'$.

\medskip
\begin{prop}
\label{MINIMUM} Let $\,D\,$ and $R$ be as above and 
let $\,f: D \to \R\,$ be a smooth, $\,K$-invariant
strictly plurisubharmonic exhaustion of $D$.
\begin{itemize}
\item [(i)] If $\,R\,$ contains the origin, then $R$ is connected and
$\,\tilde f\,$ has a unique minimum point at the origin of 
$\,\mathcal D_{\a}$.
\item [(ii)]
If $\,R\,$ does not contain the  origin, then 
$\,\tilde f $
has a unique minimum  point on the 
diagonal line $\,\{a_1= \dots =a_r\}\,$
of 
$\,\mathcal D_{\a}$. In particular $\,R\,$ is connected.
In this case $\,G/K\,$ is necessarily of tube type.
\end{itemize}
\end{prop}

\medskip
\begin{proof} 
The minimum set of  a $K$-invariant exhaustion function $f$ of $D$ intersects  $R$ in a  non-empty $T$-invariant set.
Since $R=T \cdot \exp  \mathcal D_\a K$,   a point  $H\in \mathcal D_\a $ is a minimum of  $\tilde  f$ if and only if $\exp(H)K\in R$ is a minimum of $f|_R$, the restriction 
of $\,f\,$ to   $\,R\,$.

\noindent
(i) We already observed that $R$ is connected. Assume that $\,\tilde f\,$ has a minimum point $\,H=(a_1,\, \dots , \,a_r)\,$, different from the origin.  Then the restriction  $\,f|_R\,$ of $\,f\,$ to the Reinhardt domain $\,R\,$ has a minimum point in 
$\,P=\exp (H)K$.
For $\,\varepsilon\,$ small enough there is  a holomorphic immersion
$$\, \iota:\Delta_{1+\varepsilon} \to R, \ \quad \quad z \to zP\,$$
from the disc of radius $\,1+ \varepsilon\,$ to $\,R$. The
pull-back $\,f \circ \iota\,$ of $\,f\,$ via $\,\iota\,$ is a smooth strictly subharmonic 
$\,S^1$-invariant function. Hence it has a minimum point in~$\,0\,$ and,
by construction, in $\,1$. It follows that  $\,f \circ \iota\,$ is necessarily constant,
contradicting the fact that it is strictly subharmonic.
\pn
(ii) Let $\,H=(a_1,\, \dots , \,a_r)\,$ be a minimum point of $\,\tilde f\,$. In  this case,  all $a_j$'s  are
different from $\,0\,$. As a consequence $\,\textstyle 2\coth(a_j)
\frac{\partial \widetilde f}
{  \partial a_j}(H)=0\, $, for $j=1,\ldots, r$. By $\,(iv)\,$ of Proposition \ref{LEVI}, in the non-tube case this contradicts the strict plurisubharmonicity of $f$, implying that 
the space $\,G/K\,$
is necessarily of tube type. The strict plurisubharmonicity of $\,f\,$ along with
 $\,(iii)\,$ of Proposition \ref{LEVI}, also implies that $\,a_j=a_k\,$ for every $\,j,k=1, \dots, r$.
 Hence $H$ lies on the diagonal of $\a$. The uniqueness of the minimum
 follows from standard arguments as in \cite{AzLo93} or by the the following direct argument.
 
Recall that 
 $\,\mathcal D_{\log} \,$ is convex by the Steinness of $D$.  By
 Remark \ref{LOGCONV}, the associated function $\,\hat f  \,$ has
 everywhere strictly positive definite Hessian.
In particular its restriction to the diagonal $\,\mathcal D_{\log}  \cap 
\{s_1= \dots =s_r\}\,$ is a strictly convex exhaustion function. Hence it has 
a unique minimum, implying that $\,\tilde f$ has a unique minimum on
$\,\mathcal D_{\a}  \cap 
\{a_1= \dots =a_r\}\,.$
\end{proof}

\medskip
Consider the following classes of functions:

  \medskip
\noindent
 - $C^0(\mathcal D_\a)^W$: continuous $\,W$-invariant functions on $\,\mathcal D_\a$,
 
\medskip
\noindent
 - $C^{\infty}(\mathcal D_\a)^W$ : smooth $\,W$-invariant functions on $\,\mathcal D_\a$,

 \medskip
\noindent
- $C^0(D)^K$: continuous $\,K$-invariant functions on $\,D$,

\medskip
\noindent
 - $C^{\infty}(D)^K$: smooth $\,K$-invariant functions on $\,D$.
 
 \medskip

Since the $K$-action on $D$ is proper and every $K$-orbit intersects the slice $\exp \mathcal D_\a K$ in a
$W$-orbit, it is easy to check that the map $f \to \tilde f$ is a bijection from $C^0(D)^K$
 onto $C^0(\mathcal D_\a)^W$. By Theorem 4.1 in \cite{Fle78} (see also \cite{Dad82}) such a map is also a bijection  from $C^\infty(D)^K$ onto $C^\infty(\mathcal D_\a)^W$.
Define

 \smallskip
- $LogConv^{\infty,+}(\mathcal D_\a)^W$: smooth, 
 $\,W$-invariant functions on $\,\mathcal D_\a\,$ such that the form defined in   (ii) of Proposition \ref{LEVI} is strictly positive definite,

\smallskip
- $P^{{\infty},+}(D)^K$: smooth,  $\,K$-invariant, strictly plurisubharmonic functions
(i.e.\,with strictly positive definite Levi form) on $\,D$.

\bn
  Our first  result is the following theorem.

\medskip
\begin{theorem}
\label{BIJECTIVEK1} 
Let $\,D\,$ be a Stein $\,K$-invariant domain in an
 irreducible non-compact Hermitian symmetric space  $\,G/K\,$ of rank $\,r$.
Then $\,f\,\in P^{{\infty},+}(D)^K$ if and only if $\,\tilde f \in LogConv^{\infty,+}(\mathcal D_\a)^W$.
 \end{theorem}

\medskip
\begin{proof} By (ii) of Proposition \ref{LEVI}
, if 
 $\,f\,$ is strictly plurisubharmonic on $D$, then  $\,\tilde f \in LogConv^{\infty,+}(\mathcal D_\a)^W$.

Conversely, assume that $\,\tilde f \in LogConv^{\infty,+}(\mathcal D_\a)^W$ and $\,r>1\,$.
We need to show that  
 the terms in
(iii)  and (iv) of Proposition \ref{LEVI} are strictly positive, the  ones in  (iv) occurring only in the non-tube case.

For  the terms in (iii), without loss of generality, it is sufficient to consider the case  $r=2$, and  $H=(a_1,a_2) \in \a^+$, where  $a_1\geq a_2 \geq 0$.
Assume  first  $\,a_1> a_2 > 0$.
Then $(\tanh a_1,\tanh a_2)=(e^{s_1},e^{s_2}) \in R^*$, where $R$ is the Reinhardt domain associated to $D$.
Let $d_0<0$ and $t_0>0$ be real numbers  defined by $\,(s_1,s_2)=(d_0+t_0,d_0-t_0)$. 

Denote by SSC (smooth stably convex) those smooth functions with everywhere
positive definite Hessian.
The function $\,\hat f\,$, which is invariant under coordinate  permutations,  is SSC by Remark \ref{LOGCONV}. Therefore   
 $\,g(t):=\hat f(d_0+t,d_0-t)\,$ is even and SSC. Consequently,  for $\,t_0>0$, the inequality 
$$ g'(t_0)= \textstyle \frac{\partial f}{  \partial \rho_1}(e^{d_0+t_0},e^{d_0-t_0})e^{d_0+t_0}-
\frac{\partial f}{  \partial \rho_2}(e^{d_0+t_0},e^{d_0-t_0})e^{d_0-t_0}>0\,,$$
holds true, which  combined with  formulas (\ref{DEREFFETILDE})  yields the desired result.
$$g'(t_0)= \textstyle \frac{1}{2}
\big ( \sinh (2a_j)\frac{\partial \widetilde f}{  \partial a_j}(H)-
   \sinh (2a_l)\frac{\partial \widetilde f}{  \partial a_l}(H)\big )>0\,.$$

Next we need to estimate the terms  (iii) of Proposition \ref{LEVI}, when $H$ lies on the boundary of the Weyl chamber $ \a^+$. Consider $\,H=(a,a)\,$,  with $\,a \not=0$.
Set $\tanh a= e^{d_0}$ and $(\tanh a_1,\tanh a_2)=(e^{d_0+t},e^{d_0-t})$, and 
recall  that $\,g'(0)=0$. 
Then the corresponding  term in (iii) of Proposition \ref{LEVI}  is the limit
$$\textstyle \frac{1}{2}\lim_{t \to 0}  \frac{g'(t)}{\sinh(a_1+a_2)\sinh(a_1-a_2)}
=\frac{1}{4}\lim_{t \to 0}  \frac{g'(t)}{t}  \frac{s_1-s_2}{\sinh(a_1+a_2)\sinh(a_1-a_2)}$$
$$=\textstyle  \frac{1}{4} \frac{g''(0)}{\sinh(2a)\cosh(2a)} \lim_{a_1-a_2 \to 0}   \frac{\log \tanh a_1-\log \tanh a_2}{ a_1-a_2}=g''(0)c(a)\,,$$
which is positive since  $\,c(a)\,$ is a positive real number and $g''(0)>0$ ($g$ is even and SSC). 

If $H=(a_1,0) \in \mathcal D_\a$,  with $\,a_1> 0$,  then the  Reinhardt domain
$R$ associated to $D$ is necessarily complete and the  term to be evaluated reduces to 
$$ \textstyle \frac{1}{\sinh^2 a_1}
\sinh (2a_1)\frac{\partial \widetilde f}{  \partial a_1}(a_1,0)\,.$$
Moreover  
$$\textstyle 2\coth (2a_1)\frac{\partial \widetilde f}
{  \partial a_1}(a_1,0)+\frac{\partial^2 \widetilde f}
{  \partial a_1^2}(a_1,0)>0,$$
implying that
the function $s_1 \to f(e^{s_1},0)$ is  SSC
(cf.\,Rem.\,\ref{LOGCONV}). Since $R$ is complete, then 
$\lim_{s_1 \to -\infty} f(e^{s_1},0)$ is finite. As a consequence $s_1 \to f(e^{s_1},0)$
is strictly increasing and so is $a_1 \to \tilde f(a_1,0)= f(\tanh a_1,0)$.
Hence $\frac{\partial \widetilde f}{  \partial a_1}(a_1,0)\,$
is positive, as wished. 

Finally note that for $a_1=a_2=0$ the analytic extension of
our term is given by 
$$\textstyle 2  \frac{\partial^2 \widetilde f}{  \partial a_1^2}(0,0)
=2  \frac{\partial^2 \widetilde f}{  \partial a_2^2}(0,0)\,,$$
which is  strictly  positive by  assumption.

We are left to examine the terms in (iv), which  only appear in the non-tube case. 
Our arguments are similar to the ones used above.
By  Proposition \ref{MINIMUM}, the  Reinhardt domain $R$ associated to $D$ is complete. Then 
$\,\lim_{s_j \to -\infty } \hat f(s_1,\dots,s_j,\dots\,s_r)$ is finite.
Since  $\, \hat f\,$ is SSC,  
the function $\,s_j \to \hat f(s_1,\dots,s_j, \dots \,s_r)\,$ 
is strictly increasing and so is $\,a_j 
\to \tilde f(a_1,\dots,a_j, \dots \,a_r)$. 
Hence
$$\textstyle 2\coth(a_j)
\frac{\partial \widetilde f}
{  \partial a_j}(a_1,\dots a_r)>0,\quad \hbox{ for $a_j>0$}.\,$$
The limit
$$\,\lim_{a_j \to 0} \textstyle 2\coth(a_j)
\frac{\partial \widetilde f}
{ \partial a_j}(a_1,\dots,a_j,\dots,a_r)=2
\frac{\partial^2 \widetilde f}
{  \partial a_j^2}(a_1,\dots,0,\dots,a_r)\,$$
is strictly positive  as well,  by assumption. 
\end{proof}

\bigskip
Consider the ($T \ltimes \mathcal S_r$)-action on $\Delta^r$, where 
 $\mathcal S_r$ denotes the group of coordinate permutations.
As a consequence of Proposition \ref{PSHPOLIDISC} one has the following corollary.

\medskip
\begin{cor} 
\label{PSHONR}
Let $\,D\,$ be a Stein $\,K$-invariant domain in an
 irreducible non-compact Hermitian symmetric space  $\,G/K\,$ and let $\,R\,$  be the associated 
 Reinhardt domain. 
 The map $f\to  f|_R$ is a bijection between 
 $\,P^{{\infty},+}(D)^K\,$ and  $\,P^{{\infty},+}(R)^{T \ltimes \mathcal S_r}$.
 \end{cor} 

\medskip
\begin{remark} 
\label{PSHREINAHARDT}
If  $\,R\,$ does not contain
the origin,  
 then, by  Remark \ref{LOGCONV}, the condition  $\,f\,\in P^{{\infty},+}(D)^K$ is also equivalent to requiring  that the smooth invariant function 
$\,\hat f\,$ has strictly positive definite Hessian on $\,\mathcal D_{log}$.
\end{remark}

\medskip
\begin{cor} 
\label{CHARACTER} $($see \cite{BeDa91}, Thm.\,$3'$ and Thm.\,4$)$
Let $\,D\,$ be a Stein $\,K$-invariant domain in an
 irreducible non-compact Hermitian symmetric space  $\,G/K\,$ and let $\,R\,$  be the associated 
Reinhardt domain. Then
 \begin{itemize}
\item [(i)]  If $G/K$ is of tube type, then $\,D\,$ is Stein if and only if $\,R\,$ is Stein and connected.
\item [(ii)] If $G/K$ is not of tube type, then $\,D\,$ is Stein if and only if
 $\,R\,$ is Stein and complete. In particular $R$ contains the origin and is connected.  
\end{itemize}
 \end{cor}

 \medskip
 \begin{proof} By Proposition \ref{MINIMUM},
if $\,D\,$ is Stein then the intersection $\,R =D \cap \Delta^r\,$ is Stein, connected and,
 in the non-tube case, complete. Conversely, let
  $\,R\,$ be a   Stein, connected Reinhardt domain, invariant under coordinate permutations  which,
  in the non-tube case, is also 
  assumed to be  complete.
Let $f$ be  a smooth, strictly plurisubharmonic
 exhaustion function $\,f\,$ of $\,R$. By averaging, 
  $\,f\,$ may be assumed to be  invariant with respect to $\,T\,$ and to 
 coordinate permutations.  Proposition 
 \ref{PSHPOLIDISC}  implies that the  function
 $\ \tilde f:\mathcal D_\a \to \R\,$, associated to $f$, belongs to  $\,LogConv^{\infty,+}(\mathcal D_\a)^{W}$.
By  Theorem \ref{BIJECTIVEK1},  $\ \tilde f$ extends to a smooth, $K$-invariant, strictly plurisubharmonic
 exhaustion function of $\,D$. Hence $D$ is Stein.
 \end{proof}

\medskip
\begin{remark}
\label{SCHLICHT}
{\rm 
The envelope  of holomorphy of a $K$-invariant domain  $D$ in $G/K$ is 
described, without proof,  in terms of the associate Reinhardt domain $R$ in  Theorem 5 of  \cite{BeDa91}: }

\sn
 if $G/K$ is of tube type, then
$\widehat D= K\cdot \widetilde R$, where $\widetilde R$ is the smallest connected
  Stein,  Reinhardt domain containing $R$;

\sn
if $G/K$ is not of tube type, then
$\widehat D= K\cdot \widetilde R$, where $\widetilde R$ is the smallest connected and complete Stein, Reinhardt domain containing $R$. 

\sn
{\rm One can easily prove the above theorem in the following cases. If $R$ is connected and intersects the coordinate hyperplanes,
then the envelope of holomorphy $\widehat D$ of $D$ is schlicht and
coincides with
$K\cdot \widehat R$, where $\widehat R$ is the envelope
of holomorphy of $R$.
Indeed,   by \cite{Car73}, Thm.\,2.12.
the envelope of holomorphy $\widehat R$
of $R$,  is schlicht and is a Stein, complete Reinhardt domain in $\Delta^r$.
As a consequence,
the invariant domain  $K\cdot \widehat R$ is contained in $\widehat D$, and it  is Stein by Corollary \ref{CHARACTER}. 
It follows that $\widehat D=K\cdot \widehat R$.

By the same argument, this equality holds true also in the case when $G/K$ is of tube type and $D$ is  any $\,K$-invariant domain such that $R$ is connected.}
\end{remark}

\bigskip
Our next goal is to  extend the characterization of smooth, $K$-invariant, 
strictly plurisubharmonic functions 
on $D$ obtained in  Theorem \ref{BIJECTIVEK1} to
some wider classes of $K$-invariant functions.  Namely:

\medskip
\noindent
- $P(D)^K$: plurisubharmonic, $\,K$-invariant functions on $\,D$,

\medskip
\noindent
- $P^{\infty}(D)^K$: smooth, plurisubharmonic, $\,K$-invariant functions on $\,D$,

\medskip
\noindent
- $P^+(D)^K$: functions which, on every relatively compact $K$-invariant  domain $\,C\,$
in $\,D\,$, are the sum  
  $\,g\,+\,h\,$,  for some  $g \in  P(C)^K$ and $h \in P^{\infty,+}(C)^K$.

\bigskip
In order to do that we need to define the appropriate classes of functions on the associated domain $\mathcal D_\a$: 


\mn
-  $LogConv(\mathcal D_\a, [-\infty, \infty))^W$: limits of 
 decreasing sequences  in  $LogConv^{\infty,+}(\mathcal D_\a)^W$,
 \pn

\mn
-  $LogConv^{\infty}(\mathcal D_\a)^W$: smooth functions in 
$LogConv(\mathcal D_\a, [-\infty, \infty))^W$,

\medskip
\noindent
- $LogConv^+(\mathcal D_\a, [-\infty, \infty))^W$:  
functions which, on every relatively compact $W$-invariant  domain $\,\mathcal C\,$ of $\mathcal D_\a$,  
are the  sum $\,\tilde g\,+\,\tilde h\,$ for some  $\tilde g \in  LogConv( \mathcal C, [-\infty, \infty))^W$ and $\tilde h \in LogConv^{\infty,+}(\mathcal C)^W$.

\medskip
\begin{remark} 
\label{PROPCLASSES} 
$(i)$
The class  $\,LogConv^{\infty}(\mathcal D_\a)^W\,$ 
coincides with  the family  of  smooth $\,W$-invariant functions on $\,\mathcal D_\a$ for which the form in $(ii)$ of Proposition \ref{LEVI} is 
non-negative. 
One inclusion is clear. Conversely, if $\,\tilde f\,$ is smooth and 
the form in $(ii)$ of Proposition \ref{LEVI} is non-negative, then $\,\tilde f\,$ is the limit
of the sequence $\,
\tilde f_n(a_1,\dots a_r)= \tilde f(a_1,\dots a_r)+ \frac{1}{n}\sum a_j^2$.
Hence $\tilde f$ belongs to 
$LogConv^{\infty}(\mathcal D_\a)^W$.
In particular 
$$LogConv^{\infty,+}(\mathcal D_\a)^W
\subset LogConv^{\infty}(\mathcal D_\a)^W\,.$$ 

\sn
$(ii)$
The class $P^+(D)^K$ coincides with the family of functions 
which are locally the sum of some $\,g\,$ plurisubharmonic and $\,h\,$ smooth  
strictly plurisubharmonic,
i.e. the strictly plurisubharmonic functions according to \cite{Gun90}, Def.$\,$1, Sect. L, p. 118.
Indeed, assume that $f$ is strictly plurisubharmonic according to   such definition. Choose a $K$-invariant, smooth 
strictly plurisubharmonic function $\psi$ on $D$ and let $C$ be a
relatively compact $K$-invariant  domain of $D$.
Then there exists $\varepsilon>0$
such that $f-\varepsilon \psi$ is plurisubharmonic on
$C$. That is,  $\,f= g +\varepsilon \psi$, with $g$ psh
and $K$-invariant on $C$.
\end{remark}

\bigskip
The following lemma shows that all functions in the above classes 
are continuous.

\medskip
\begin{lem} 
Let  $R$ be a Reinhardt domain.
\begin{itemize} 
\item [i)] Any $T$-invariant plurisubharmonic function $f$ on $R$ is continuous. Its pluripolar set is the union of the intersections of $R$ with some
coordinate subspaces.
\item [ii)]
The class $LogConv(\mathcal D_\a, [-\infty, \infty))^W$ is contained in $\ C^0(\mathcal D_\a, [-\infty, \infty))^W$.
\end{itemize}
 \end{lem}

 \medskip
\begin{proof} (i) First consider the case $r=1$. On $R^*=R \setminus \{0\}$ one has $f(z)=\hat f(\log |z|)$,
with $\hat f$ convex. Hence the restriction of $f$ to $R^*$ is continuos. If $0 \in R$,
by the upper semicontinuity of subharmonic functions, one has  $f(0)= \limsup_{z \to 0}f(z)$. 
Assume by contradiction that 
$$ \liminf_{z \to 0}f(z)<f(0) \,.$$
Then there exists $z_1 \in R$ close to the origin such that 
$f(z_1) <f(0)$. By the submean value property  and the $S^1$-invariance of $f$  one has 
$$\textstyle f(0)\leq\frac{1}{2\pi} \int_0^{2\pi} f(e^{i\theta}z_1)d\theta =f(z_1)<f(0)\,,$$
 which is a contradiction.

For $r=2$, an argument analogous to the above one shows that $f$ is continuos on
$R^*= R \cap (\Delta^*)^r$.  We now prove continuity on the coordinate hyperplanes $\{(z,w) \in R \ : \ zw=0\,\}$
(on each hyperplane it can be constant and equal to $-\infty$).
Assume by contradiction that there exists  $(z_0,0) \in R$ such that 

\begin{equation}\label{ASSURDO}
 \limsup_{(z,w) \to (z_0,0)}f(z,w)- \liminf_{(z,w) \to (z_0,0)}f(z,w)> \varepsilon >0\,.
\end{equation}
Since $ \limsup_{(z,w) \to (z_0,0)}f(z,w)=f(z_0,0)$ by plurisubharmonicity and  $f$ is continuous 
on the hyperplane $w=0$, there exists 
a neighborhood $B$ of $z_0$ in $\C$ such that 
$$\limsup_{(z,w) \to (z_0,0)}f(z,w)- \varepsilon/2=f(z_0,0)-\varepsilon/2<f(\zeta,0)\,,$$
for every $\zeta \in B$.
By (\ref{ASSURDO}), we can choose $(\zeta_1,w_1)$ close to $(z_0,0)$ such that $\zeta_1 \in B$ and 
$$f(\zeta_1,w_1)< \liminf_{(z,w) \to (z_0,0)}f(z,w)+\varepsilon/2 <\limsup_{(z,w) \to (z_0,0)}f(z,w)- \varepsilon/2.$$
Then, by the submean value property for subharmonic functions and by the $T$-invariance of $f$  one has 
$$\textstyle 
\limsup_{(z,w) \to (z_0,0)}f(z,w)- \varepsilon/2<f(\zeta_1,0)\leq
\frac{1}{2\pi} \int_0^{2\pi} f(\zeta_1,e^{i\theta}w_1)d\theta =$$
$$=f(\zeta_1,w_1)<\limsup_{(z,w) \to (z_0,0)}f(z,w)- \varepsilon/2\,.$$
giving a contradiction.

The above argument also
shows that the pluripolar set of $f$ consists of either the origin, or the intersection of $R$ with one of the the coordinate lines 
$\,\{z=0\}$, $\,\{w=0\}\,$ or with both of them.

Now we can proceed inductively and obtain the statement for $\,r>2$.

(ii) By Theorem \ref{BIJECTIVEK1}, to a decreasing sequence $\tilde f_n$ of functions in 
$LogConv^{\infty,+}(\mathcal D_\a)^W$ there corresponds  a decreasing sequence $f_n$ in 
$P^{\infty,+}(D)^K$, whose limit $f$ necessarily belongs to $P(D)^K$. The restriction $f|_R$ of 
$f$ to $R$ is a plurisubharmonic $T$-invariant function. So (i) implies that 
 $f|_R$ is continuous and consequently so is the corresponding $\tilde f$
 in $LogConv(\mathcal D_\a, [-\infty, \infty))^W$, 
 which is the limit of the~$\tilde f_n$.
 \end{proof}

  \nmedskip
 Summarizing,  the following inclusions hold true 
 
 $$ 
\begin{matrix}
LogConv^+(\mathcal D_\a, [-\infty,\infty))^W  & \subset  & LogConv(\mathcal D_\a, [-\infty,\infty))^W &
     \subset &C^0(\mathcal D_\a,[-\infty,\infty))^W \cr
    \cup  &   &  \cup & &  \cup \cr
LogConv^{{\infty},+}(\mathcal D_\a)^W  & \subset  & LogConv^{\infty}(\mathcal D_\a)^W & \subset &C^\infty(\mathcal D_\a)^W
\end{matrix} $$

\mn
and our  complete result is stated in  the next  theorem.

\medskip
\begin{theorem}
\label{BIJECTIVEK} Let $D$ be a Stein $K$-invariant domain in an irreducible non-compact
Hermitian symmetric space $G/K$.
The map $f\to \tilde f$ is a bijection between
the following classes of functions 
\begin{itemize}
\item [(i)] $\,P^{{\infty},+}(D)^K\,$ and $\ \ \,LogConv^{{\infty},+}(\mathcal D_\a)^W$,
\item [(ii)] $\,P\,(D)^K\,\ \ \ \,$ and $\, \ \  \,LogConv(\mathcal D_\a, [-\infty,\infty))^W$,
\item [(iii)] $\,P^{\infty}(D)^K\,\ \ $ and $\ \ \,LogConv^{\infty}(\mathcal D_\a)^W$,
\item [(iv)]  $\,P^{+}(D)^K\,\ \ $ and $\ \ \,LogConv^+(\mathcal D_\a, [-\infty,\infty))^W$.
\end{itemize}

\noindent
In particular $K$-invariant plurisubharmonic functions on $D$ are continuous.
\end{theorem}

\medskip
\begin{proof} (i) is the content of Theorem   \ref{BIJECTIVEK1}.
  By averaging over $K$,  a  $K$-invariant,  plurisubharmonic  function on $\,D\,$ is the decreasing limit of smooth $K$-invariant,   strictly plurisubharmonic   functions 
  (cf.\,\cite{Gun90},\,Sect.\,K). Then the assert (ii) follows from~(i).
As smooth $K$-invariant  functions on $D$ correspond to smooth $W$-invariant functions on $\,\mathcal D_\a$,   the previous argument also proves   statement (iii).
Finally (iv) follows from the   definitions of 
$LogConv^{+}(\mathcal D)^W\,$ and $P^{+}(D)^K$,  by
averaging all the involved functions over $K$.
\end{proof} 

\medskip
Let $T \ltimes \mathcal S_r$ act on $\Delta^r$ as in Corollary \ref{PSHONR}. The previous theorem can be reformulated as follows.

\medskip
\begin{theorem}  
\label{PLURIREIN} Let $\,D\,$ be a Stein $\,K$-invariant domain in an
 irreducible non-compact Hermitian symmetric space  $\,G/K\,$ and let $\,R\,$  be the associated 
 Reinhardt domain. 
The map $f\to  f|_R$ is a bijection between 
\begin{itemize}
\item [(i)] $\,P^{{\infty},+}(D)^K\,$ and  $\,\ \ \ P^{{\infty},+}(R)^{T \ltimes \mathcal S_r}$,
\item [(ii)] $\,P\,(D)^K\,\ \ \ \,$ and $\,\ \ \ P\,(R)^{T \ltimes \mathcal S_r}$,
\item [(iii)] $\,P^{\infty}(D)^K\,\ \ $ and $\,\ \ \ P^{\infty}(R)^{T \ltimes 
\mathcal S_r}$,
\item [(iv)]  $\,P^{+}(D)^K\,\ \ $ and $\,\ \ \ P^{+}(R)^{T \ltimes 
\mathcal S_r}$.
\end{itemize}
\end{theorem}
 

\bigskip
\section{Appendix: a $K$-invariant potential for the Killing metric: the logarithm of the Bergman kernel function.}
\label{OMEGAK}


\bigskip
Let $G/K$ be an irreducible non-compact Hermitian symmetric space. The
Killing form $B$ of $\g$, restricted to $\p$,  induces  a $G$-invariant
K\"ahler metric $h$ on $G/K$. 
In this section  we exhibit a $K$-invariant
potential $\,\rho\,$ of the Killing metric $\, h$ in a Lie theoretical fashion. As such a $K$-invariant
potential is unique up to an additive constant (Rem.\,\ref{UNIQUE}),
 then,  up to an additive constant, it coincides with the logarithm of the  Bergman kernel function (Cor.\,\ref{BERGMAN}).

\bn

In order to define  $\,\rho$, recall  the  decomposition
$\,G=K \exp \a\, K\,$  and write an element of $\,G/K\,$ as  $\,kaK$, where 
$\,k \in  K\,$ and $\,a=\exp H\,$,
with  $H=\sum_ja_jA_j\in\a$.

\bigskip
\begin{prop}
\label{POTENTIALK} 
 Let  $\,\widehat\rho\,$ be a real valued function satisfying  
$\, \widehat\rho\,'(t)= \frac{\cosh t-1}{\sinh t} \, $. Then 
\item{$(i)$} the $\,K$-invariant function $\,\rho:G/K \to \R\,$ defined by 
$$\, \textstyle \rho(ka K):=  \textstyle \frac{1}{4}{\sum_{j=1}^r}
\widehat\rho(2a_j)B(A_j,\,A_j)\,,$$
is a potential of  the Killing metric  $\,h$;

\item{$(ii)$} the moment map $\,\mu: G/K \to \k^*\,$ associated with
$\,\rho\,$ is given by 
$$\,\mu(kaK)(X) =\textstyle {\frac{1}{2}\sum_{j=1}^r} 
\sinh (2a_j)\widehat\rho\,'(2a_j)B (\Ad_{k^{-1}} C,K^j)\,,$$
where  $X\in\k$.
\end{prop}

\medskip
\begin{proof}

\smallskip
\noindent
(ii)  Resume the notation of Proposition \ref{LEVI}. It was shown in the proof 
 Proposition \ref{LEVI} that for $z=aK$ one has
$$\,d^c \rho (\widetilde C_z)=0\,,  $$
for  all $C\in \m \oplus  \bigoplus_{\alpha\in \Sigma^+\atop\alpha\not=2e_j}\k[\alpha]$.   
Let $K^j\in\k[2e_j]$ be defined as in (\ref{KJPJ}). From (\ref{UNO}) it follows that 
\begin{equation}\label{MOME} 
d^c \rho  ( \widetilde{ K^j}_z)=
\textstyle -\frac{1}{2}{\sinh (2a_j)}\widehat\rho\,'(2a_j)B(A_j,\,A_j)\,.
\end{equation}
As $B(A_j,A_j)=-B(K^j,K^j)$, 
then (\ref{MOME}) and (\ref{DCf}) imply (ii).

\medskip
\noindent
(i) Define
$h_\rho(\,\cdot\,,\,\cdot\,)=-dd^c \rho (\,\cdot\,,I_0\,\cdot\,)$. Because of the invariance of $\rho$ and of the orthogonality relations proved in
Proposition \ref{LEVI}, 
it is sufficient to   prove  that
$
 h_\rho (a_* P ,a_* Q)= B(P,Q)
$
for  $P$, $Q$ both in
one of the blocks $a_*\a$, $a_*\p[e_j+e_l]$ and $a_*\p[e_j]$.

\bigskip
\noindent
 {\bf  \boldmath The Hermitian form $h_\rho$ on $a_*\a$.}
 Let $A_j, A_l \in\a$, be as in (\ref{NORMALIZ1}). Then, by
(\ref{TILDEFIELDS2}) and (\ref{DIFFER}), 
 $$h_\rho(A_j,A_l)=- dd^c\rho(a_*A_j
,\,a_*I_0A_l)=-dd^c\rho(a_*P^l,a_*A_j)$$
 $$\textstyle 
 = {1\over \sinh
(2a_l)}dd^c\rho((\widetilde{K^l})_z,(\widetilde{A_j})_z)=
-{1\over \sinh(2a_l)} \dds \mu^{K^l }(\exp sA_j \, z),$$
where  $P^l \in\p[2e_l]$ and $K^l\in \k[2e_l]$ are defined in
(\ref{KJPJ}). By (ii) and (\ref{MOME}) such quantity vanishes if $l\not=j$. For 
for $j=l$, it becomes
$$\textstyle 
{1\over \sinh (2a_j)}{1\over 2} \dds 
{\sinh (2a_j+2s)}\widehat\rho\,'(2a_j+2s))B(A_l,\,A_l)=
$$
$$\textstyle 
{1\over \sinh (2a_j)} \big(
{\cosh (2a_j)}\widehat \rho\,'(2a_j)+
{\sinh(2a_j)}\widehat \rho\,''(2a_j)
\big) 
B(A_l,\,A_l)=B(A_l,\,A_l)\,,
$$
where the last equality follows from the assumption
 $\, \widehat \rho'(t)= \frac {\cosh t-1}{\sinh t}$.

\bigskip
Recall that if $\alpha \in \Sigma^+\setminus\{2e_j\}$, then
$I_0\p[\alpha]=\p[\beta]$, for some $\beta \not=0$.
Let   $P \in \p[\alpha]$, with $\alpha$  as above. Write
$\,P=X^\alpha-\theta X^\alpha\in\p[\alpha],$ and $\,Q=I_0P =X^\beta-\theta
X^\beta \in \p[\beta]$. 
Define $K:=X^\alpha+\theta X^\alpha \in \k[\alpha]$  and $C
:=X^\beta+\theta X^\beta \in \k[\beta]$. 
As $K^j=[I_0A_j,A_j]$,  by (\ref{FORMULONE}) and (ii), 
  $$h_\rho(a_*P,a_*P )= \textstyle
-\frac{1}{ \sinh \alpha(H) \sinh \beta(H)}{1\over 2} \sum_k  {\sinh (2a_k)} \widehat\rho\,'(2a_k)
B \big ([C,K] ,\,[I_0A_k,\,A_k]\big)\,=$$
$$ \textstyle
= - \frac{1}{ \sinh \alpha(H) \sinh \beta(H)}{1\over 2} \sum_k  \sinh (2a_k) \widehat\rho\,'(2a_k)
B \big (K ,\,[[I_0A_k,A_k],\,C ] \big)\,.$$

\bn
From the  Jacobi identity, one has
$$
B \big (K  ,\,[[I_0A_k,A_k],\,C ] \big)=-B\big (K  ,\,[[C
,\,I_0A_k],\,A_k]+[[A_k,\,C ],\,I_0A_k]  \big)=$$
$$=B\big ([A_k,\,K] ,\,[I_0A_k,\,C]\big)-   
B\big ([I_0A_k,\,K],\, [A_k,\,C]  \big)=$$
$$\alpha(A_k)B\big (P,\,I_0[A_k,\,C]\big)-   
\beta(A_k)B\big (I_0[A_k,\,K],\, Q \big)=$$
$$=\big (\alpha(A_k)\beta(A_k)+
\beta(A_k) \alpha(A_k)\big) B\big (P ,\, I_0Q   \big)\,.$$
As $I_0Q=-P$,
one  obtains 
$$h_\rho(a_*P,a_*P)= 
 \textstyle \frac{1}{2 \sinh \alpha(H)\sinh \beta(H)}\sum_k
\sinh (2a_k)\widehat\rho\,'(2a_k)\big(
\alpha(A_k)\beta(A_k)+\beta(A_k)\alpha(A_k)\big) B(P ,\,P).$$

\bn
 We are left to check the following  cases.

\bigskip
\noindent
 {\bf  \boldmath The Hermitian form $h_\rho$ on $a_*\p[e_j+ e_l] $.}
\pn
Here   $\, \alpha =e_j+e_l\,$ and  $\, \beta=e_j-e_l$.  Then for
$P\in\p[e_j+e_l]$, one has 
  $$h_\rho(a_*P,a_*P)= 
  \textstyle
  {1\over {2\sinh (a_j+a_l) \sinh (a_j-a_l)}} \big ( \sinh (2a_j)\widehat\rho\,'(2a_j)
- \sinh (2a_l)\widehat\rho\,'(2a_l) 
\big )
B( P ,\,P)=$$
$$ 
\textstyle
\frac{ \cosh (2a_j)-\cosh(2a_l)}{2\sin(a_j+a_l) \sin (a_j-a_l)}B( P ,\, P
)=B( P ,\, P ),$$
due to the identity
 $\,\cosh (2a_j)-\cosh (2a_l)=2\sinh (a_j+a_l)\sinh(a_j-a_l)$. 

 \bigskip
\noindent
 {\bf  \boldmath The Hermitian form $h_\rho$ on $a_*\p[e_j]$.}
\pn
Here   $\, \alpha=\beta =e_j$. 
Then for $P\in\p[e_j]$, one has 
  $$h_\rho(a_*P,a_*P)= 
\textstyle
{1\over {2\sinh^2( a_j)}}  \sinh (2a_j)\widehat\rho\,'(2a_j)B(P ,\,P )$$
 $$\textstyle= {1\over { 2\sinh^2( a_j)}}(\cosh (2 a_j)-1)B( P ,P )=B( P ,P ). $$
 This concludes the proof of (i).
\end{proof}

\bigskip
The following remark shows that,
up to an additive constant, the $K$-invariant potential $\rho$ 
of
the Killing metric $h$  is unique.

\bigskip
\begin{remark} 
\label{UNIQUE}
Let $\rho_1$ and $\rho_2$ 
be smooth $K$-invariant functions on $G/K$ such that
$dd^c\rho_1=dd^c\rho_2$. Then $\rho_1-\rho_2$ is constant.
\end{remark}

\begin{proof}
As $\rho_1-\rho_2$ is pluriharmonic and $G/K$ 
is  contractible, there exists a unique holomorphic function 
$f : G/K\to \C$  such that Re$f=\rho_1-\rho_2$ (cf. \cite{Gun90}, Sect. K).
By $K$-invariance, one has Re$f\,=$\,Re$f\,\circ \, k $, for every $k \in K$, 
and $f \circ k - f \equiv \lambda(k)$, where $\lambda(k)$ is a constant in $i\R$.
Moreover, given $h,\ k \in K$ and $z \in G/K$, one has
$$f (z)+\lambda(hk)=f((hk) \cdot z)=f(k \cdot z)+\lambda(h)
 =f(z)+\lambda(h) +\lambda(k)\,.$$
Hence the map $\lambda: K \to i\R$ is a Lie group homomorphism and it is necessarily trivial by the compactness of $K$. Thus $f$ is $K$-invariant and 
it is  is also invariant with respect to the induced local 
 $K^\C$-action on $G/K$. Since $K^\C$ acts locally transitively on an open subset of 
 $G/K$ (cf. \cite{Wol72}),  the holomorphic function  $f$  is constant and so is its real part $\rho_1-\rho_2$.
\end{proof}

\bigskip
Since the logarithm of the Bergman kernel function is a $K$-invariant potential of
the Killing metric  (see \cite{KoNo69}, Vol.2, Exa. 6.6 p. 162 and Thm. 9.6 p. 262),
 one can draw the following conclusion. 

\bigskip
\begin{cor}  
\label{BERGMAN}
Up to an addictive constant, the smooth $K$-invariant exhaustion function $\rho$ coincides with the logarithm of the Bergman kernel function.
\end{cor}

\medskip
\begin{exa}  
\label{DISC}
As an example, consider the unit disc $\Delta=G/K, $ where  $G=SU(1,1)$ acts on
$\,\Delta\,$ by linear fractional transformations.   
~ Fix the basis of $\g$, normalized as in $(\ref{KJPJ}):$
$$K^1= \begin{pmatrix}
i  &\ 0  \cr
    0 &  -i
\end{pmatrix} , \quad A_1= \begin{pmatrix}
0  & 1  \cr
    1 &  0
\end{pmatrix}, \quad P^1= \begin{pmatrix}
0  & -i  \cr
  i &  \ 0
\end{pmatrix}.
$$
Then $\, \exp a_1A_1 K = \tanh a_1  = |z|$. Choose  
$\widehat \rho(t)= -\ln \frac{1}{\cosh t +1}$, satisfying $\, \widehat\rho\,'(t)= \frac{\cosh t-1}{\sinh t} \, $. Since  
$B(A_1,A_1)=8$, then up to an addictive constant,  the logarithm of the Bergman kernel function is given by 
$$\textstyle \rho(\exp a_1A_1 K)=-
\frac {1}{4}\log \frac{1}{\cosh 2a_1 +1} B(A_1,A_1)=$$
$$\textstyle
-2\log \frac {\cosh^2a_1 - \sinh^2a_1}{2\cosh^2a_1} =
-2\log (1-|z|^2)+const \,.$$
\end{exa}


\end{document}